\documentclass{amsart}

\usepackage{amsfonts,amsmath,amssymb,amsthm,latexsym,verbatim,amscd}
\usepackage{graphics}
\usepackage{bbm}

\newcommand{\bbC}{\mathbb{C}}

\newcommand{\bbM}{\mathbb{M}}
\newcommand{\bbN}{\mathbb{N}}

\newcommand{\bbP}{\mathbb{P}}

\newcommand{\bbR}{\mathbb{R}}
\newcommand{\bbS}{\mathbb{S}}
\newcommand{\bbT}{\mathbb{T}}

\newcommand{\bbZ}{\mathbb{Z}}

\newcommand{\calL}{\mathcal{L}}

\newcommand{\calP}{\mathcal{P}}

\theoremstyle{plain}
\numberwithin{equation}{section}

\def\a{{\alpha}}

\def \l{\lambda}

\def \o{\omega}
\def \g{\gamma}

\def\P{{\mathbb P}}

\def \ch{ {\cosh } \, \o}

\def \th{ {\tanh } \, \o}

\def\un{{\mathbbm 1}}
\def\mi{{\mathbbm i}}

\newcommand{\ov}[1]{\overline{#1}}

\def\si{\sigma}

\newcommand{\R}{{\mathbb R}}
\newcommand{\SY}{{\mathbb S }}

\newcommand{\N}{{\mathbb N}}
\newcommand{\C}{{\mathbb C}}
\newcommand{\Z}{{\mathbb Z}}

\newcommand{\pk}{{\mathcal{P}}_g}
\newcommand{\pkg}[1]{\Lambda_{-1}^{#1}\Sl(\C)}

\newcommand{\SL}{\mathrm{SL}_{\mbox{\tiny{$2$}}}}
\newcommand{\Sl}{\mathfrak{sl}_{\mbox{\tiny{$2$}}}}
\newcommand{\SU}{\mathrm{SU}_{\mbox{\tiny{$2$}}}}
\newcommand{\su}{\mathfrak{su}_{\mbox{\tiny{2}}}}

\newcommand{\rmd}{{\rm d}}

\theoremstyle{plain}
\newtheorem{theorem}{Theorem}[section]
\newtheorem*{theorem*}{Theorem}
\newtheorem{corollary}[theorem]{Corollary}
\newtheorem*{corollary*}{Corollary}
\newtheorem{proposition}[theorem]{Proposition}
\newtheorem*{proposition*}{Proposition}
\newtheorem{lemma}[theorem]{Lemma}
\newtheorem*{lemma*}{Lemma}

\newtheorem*{example*}{Example}
\newtheorem{definition}[theorem]{Definition}
\newtheorem*{definition*}{Definition}

\newtheorem*{notation*}{Notation}
\newtheorem{remark}[theorem]{Remark}
\newtheorem*{remark*}{Remark}

\title{Finite type minimal annuli in $\SY^2 \times \R$}

\author{L. Hauswirth}
\address{L. Hauswirth, Universit\'e Paris-Est, LAMA (UMR 8050), UPEMLV, UPEC, CNRS, F-77454, Marne-la-Vall\'ee, France.}
\email{hauswirth@univ-mlv.fr}

\author{M. Kilian}
\address{M. Kilian, Department of Mathematics, National University of Ireland, University College Cork, Ireland.}
\email{m.kilian@ucc.ie}

\author{M. U. Schmidt}
\address{M. U. Schmidt, Institut f\"ur Mathematik, Universit\"at Mannheim, A5, 6, 68131 Mannheim, Germany.}
\email{schmidt@math.uni-mannheim.de}

\begin{document}

\begin{abstract} We study minimal annuli in $\SY^2 \times \R$ of finite type by relating them to harmonic maps $\bbC \to \bbS^2$ of finite type. We rephrase an iteration by Pinkall-Sterling in terms of polynomial Killing fields. We discuss spectral curves, spectral data and the geometry of the isospectral set. We consider polynomial Killing fields with zeroes and the corresponding singular spectral curves, bubbletons and simple factors. We investigate the differentiable structure on the isospectral set of any finite type minimal annulus. We apply the theory to a 2-parameter family of embedded minimal annuli foliated by horizontal circles.
\end{abstract}
\thanks{\it{L. Hauswirth was partially supported by the ANR-11-IS01-0002 grant.}}
\maketitle

\section*{Introduction}

In the last decade there has been an interest in extending minimal surface theory to the target space $\SY^2 \times \R$ \cite{rosenberg2002,MRstable,MR, hauswirth2006, hoffman-white2011}. A minimal surface that is conformally immersed in $\SY^2 \times \R$ is essentially described by a harmonic map $G: \Omega \subset \C \to \SY^2$. There is an important subclass of such harmonic maps that have an algebraic description, the so-called harmonic maps of finite type. For example, all constant mean curvature ({\sc{cmc}}) tori in $\R^3,\,\mathbb{S}^3,\,\mathbb{H}^3$ have harmonic Gauss maps of finite type \cite{Hit:tor, PinS}. The description of conformally immersed proper minimal annuli of finite type is analogous to the well known theory of constant mean curvature tori of finite type \cite{Hit:tor, PinS, Bob:tor, BurFPP,FerPPS, ForW, DorH:per, McI:tor}. Thus periodic minimal immersions $X : \C \to \SY^2 \times \R$ of finite type can be described by algebraic data. Up to some finite dimensional and compact degrees of freedom the immersion is determined by the so-called {\emph{spectral data}} $(a,\,b)$. This consists of two polynomials of degree $2g$ respectively $g+1$ for some $g \in \bbN$. The polynomial $a(\l)$ encodes a hyperelliptic Riemann surface called the {\emph{spectral curve}}. The genus of the spectral curve is called {\emph{spectral genus}}. The other polynomial $b(\l)$ encodes the extrinsic closing conditions. This correspondence is called the algebro-geometric correspondence. In particular, we characterize those algebraic curves which are the spectral curves of minimal annuli in $\SY^2\times \R$. The spectral data is the starting point for a deformation theory to be used in \cite{hks2}.

Let us briefly outline this paper.  In the first section we provide a short introduction to the local theory of minimal surfaces in $\SY^2 \times \R$ and prove a Sym-Bobenko type formula. In the second section we discuss polynomial Killing fields, how these are related to the Pinkall-Sterling iteration and define the spectral curve. In the third section we invoke the Iwasawa factorization of the underlying loop groups, mention the generalized Weierstrass representation, and briefly discuss the Symes method \cite{Symes_80, BurP_adl, BurP:dre}. In the fourth section we treat the spectral curve, the isospectral action, and prove some properties of the isospectral set. In particular we discuss singular spectral curves and how to remove singularities with the isospectral action. The isospectral set $I(a)$ of a minimal annulus consists of all minimal annuli with the same spectral data $(a,b)$. We decompose $I(a)$ with the help of a group action into orbits and identify each orbit smoothly with a commutative Lie group. These orbits have different dimensions. The lower dimensional orbits are in the closures of the higher dimensional ones. Our main results are
\begin{enumerate}
\item The isospectral sets $I(a)$ are compact.
\item If the 2g-roots of $a$ are pairwise distinct, then there is one orbit diffeomorphic to
\[
I(a) \cong (\bbS^1)^g.
\]
 \item If $a$ has a double root $\alpha_0 \in \bbS^1$, then there is a diffeomorphism
\[
  I(a) \cong I(\tilde{a})\quad\mbox{with}\quad a(\l)=
\bar{\alpha}_0(\l-\alpha_0)^2\tilde{a}(\l).
\]
 \item If $a$ has a double root $\alpha_0 \notin \bbS^1$ there are two invariant submanifolds diffeomorphic to
\begin{equation*} \begin{split}
&I(a)\cong I(\tilde{a})\,\cup\, G\times I(\tilde{a})
\quad \text{with}\quad G=\C\mbox{ or }\C^\times  \\&\mbox{and}\quad
a(\l)=(\l-\alpha_0)^2(1-\l \bar \alpha_0)^2 \tilde a(\l)\,.
\end{split}
\end{equation*}
\end{enumerate}
In the fifth section we turn to periodic finite type harmonic maps, spectral data and how the closing conditions are encoded in the spectral curve. In section 6 we encounter bubbletons, simple factors and prove a factorization theorem for polynomial Killing fields with zeroes. In particular we show that changing the line in the simple factor preserves the period, and that for some choices of lines the nodal singularity on the spectral curve disappears. In the last section we compute the spectral data of minimal annuli in $\SY^2 \times \R$ that are foliated by circles. These low spectral genus examples will play an important role in a forthcoming paper \cite{hks2}. This paper presents a mostly self-contained account of the integrable systems approach to minimal surfaces in $\SY^2 \times \R$ of finite type, but may also serve as a companion for other integrable surface geometries.

\section{Minimal surfaces in $\SY^2 \times \R$}
We study conformal minimal immersions $X: \Omega \subset \C\to \SY^2 \times \R$ where $\Omega$ is a simply connected domain of $\C$. We write $X = (G,\,h)$ for $G:\Omega \to \SY^2$ and $h:\Omega \to \R$, and call $G$ the{\bf{ horizontal}}, and $h$ the {\bf{vertical}} components of $X$. If we denote by $(\C,\si^2 (u)|\rmd u|^2)$ the complex plane with the metric induced by the stereographic projection of $\SY^2$ ($\si^2 (u)=4/(1+ |u|^2)^2$), the map $G:\Omega \to \C$ satisfies
\begin{equation*}
    G_{z \bar{z}}+ \frac{2 \bar G}{1+|G|^2} G_z G_{\bar{z}}=0\,.
\end{equation*}
The holomorphic {\it quadratic Hopf differential} associated to the harmonic map $G$ is given by
\begin{equation*}
    Q(G)=(\si \circ G(z) )^2 G_z {\bar G}_z ( \rmd z)^{2}:= \phi (z) (\rmd z)^2\,.
\end{equation*}
The function $\phi$ depends on $z$, whereas $Q(G)$ does not.

By conformality the induced metric is of the form $\rmd s_{}^2= \rho ^2(z) \vert  \rmd z \vert^2$, and writing $z=x + \mathbbm{i}\,y$ the partial derivatives satisfy $\vert X_x \vert^2   = \vert X_y \vert ^2$ and $X_x \perp X_y$.
Conformality reads
\[
    \vert G _x \vert^2_{\si} + (h_x)^2  = \vert G _y \vert ^2_{\si} + (h_y)^2 \quad \mbox{and} \quad
    \left<G_x ,\,G _y \right>_{\si}+ (h_x)(h_y) = 0\,,
\]
hence $(h_z)^{2}(\rmd z)^2=-Q (G )$. The zeroes of $Q$ are double, and we can define $\eta$ as the holomorphic 1-form $\eta =\pm 2i \sqrt{Q}$. The sign is chosen so that
\begin{equation*}
    h= {\rm Re} \int  \eta\,.
\end{equation*}
The unit normal vector $n$ in $\SY^2 \times \R$ has third coordinate
$$
    \langle n ,\,\tfrac{\partial}{\partial t}\rangle = n_{3} = \frac{\vert  g\vert^2-1}{\vert g\vert^2+1} \quad \mbox{where} \quad
    g^{2}:=-\dfrac{G_z}{\ov{G _{\bar z}}}\,.
$$
We define the real  function $\o : \C \to \R$  by
$$
    n_{3}:=\tanh \o\,.
$$
We express the differential $\rmd G$ independently of $z$ by
\begin{equation*}
    \rmd G = G_{\bar z} \rmd \bar z+G_{z} \rmd z =
    \frac{1}{2 \si \circ G}\overline{ g^{-1} \eta} - \frac{1}{2
    \si \circ G} g \,\eta\,,
\end{equation*}
and the metric $\rmd s^2$ is given in a local coordinate $z$ by
\begin{equation*}
    \rmd s^2=(\vert G_{z} \vert_\si + \vert G_{\bar z} \vert_\si ) ^2 \vert \rmd z\vert ^2=\frac14(\vert g\vert^{-1}+\vert g \vert)^{2}\vert \eta \vert ^{2}=
    4 \cosh ^2 \o  |Q|\,.
\end{equation*}
We remark that the zeroes of $Q$ correspond to the poles of $\o$, so that the immersion
is well defined. Moreover the zeroes of $Q$ are points where the tangent plane is horizontal. The Jacobi operator is
$$
    \calL = \frac{1}{4|Q|\, \cosh ^2 \o}\left(\partial^2_{x}+\partial^2_{y} +{\rm Ric}(n) + \vert dn \vert ^2 \right)
$$
and can be expressed in terms of $Q$ and $\o$ by
\begin{equation} \label{jacobi}
    \calL = \frac{1}{4|Q|\,\cosh ^2 \o}\left(\partial^2_{x}+\partial^2_{y} +4|Q|+\frac{2 | \nabla \o |^2}{\cosh^2 \o} \right)\,.
\end{equation}
Since $n_3= \tanh \o$ is a Jacobi field obtained by vertical translation in $\SY^2 \times \R$, we have $\calL \tanh \o=0$ and
$$
    \Delta \o +4|Q| \sinh( \o) \cosh (\o)=0\,,
$$
where $\Delta = \partial^2_{x}+\partial^2_{y}$ is the Laplacian of the flat metric.
\subsection*{Minimal annuli.} Consider a minimal annulus $A$ properly immersed in $\SY^2 \times \R$. If $A$
is tangent to a horizontal plane $\{x_3 = t\}$, the set $A \cap \{x_3=t\}$ defines on $A$
a set of analytic curves with isolated singularities at points where the tangent plane of $A$ is horizontal. Near such a singularity $q$, there are $2k+2$ smooth branches meeting at equal angles, for some integer $k\geq 1$.

Annuli are transverse to every horizontal plane $\SY^2 \times \{t\}$. To see that we claim that $A\setminus\{x_3=t\}$ defines at least three connected components. At most two of them are non-compact because the immersion is proper. Hence there is a compact disk in $A$ with boundary in $\{x_3=t\}$, a contradiction to the maximum principle. To prove the claim consider $A_1,\,A_2,\,...,\,A_{2n}$ distinct local components at $q$ of $A\setminus\{x_3=t\}$. We know that the $A_i$ alternate between $\{x_3 \geq t\}$ and $\{x_3 \leq t\}$. If $A_1$ and $A_3$ are not in the same component of $\{x_3 \geq t\}$, then $A_2$ yields a third component and $A\setminus\{x_3=t\}$ has at least three connected components. If $A_1$ and $A_3$ are in the same component of $\{x_3 \geq t\}$ we can construct a cycle $\alpha_{13}$ in $\{x_3 \geq t\}$ which meets $\SY^2 \times \{t \}$ only at $q$. We consider $\alpha_0$ in $\{x_3 >0\}$ joining a point $x$ of $A_1$ and $y$ of $A_2$. Then join $x$ to $y$ by a local path $\alpha_1$ in $A$ going through $q$. Let $\alpha_{13}=\alpha_0 \cap \alpha_1$. If $A_2$ and $A_4$ were in the same component we could find a cycle $\alpha_{24}$ on $A$ which meets $\alpha_{13}$ in a single point, which is impossible since the genus of $A$ is zero. If $A_2$ and $A_4$ are not in the same component of $\{x_3 \leq t\}$, then $A_1$ yields a third component.

Hence properly immersed annuli are transverse to horizontal planes and the third
coordinate map $h : A \to \R$ is a proper harmonic function on each end of $A$ with $dh \neq 0$. Then each end of $A$ is parabolic and the annulus can be conformally parameterized by $\C / \tau \Z$. We will consider in the following conformal minimal immersions $X: \C  \to \SY^2 \times \R$ with $X(z+\tau)=X(z)$.

Since $dh \neq 0$, the Hopf differential $Q$ has no zeroes. If $h^*$ is the harmonic
conjugate of $h$, we can use the holomorphic map $\mi(h+\mi\,h^*) : \C \to \C$ to parameterize
the annulus by the conformal parameter $z=x+\mi y$. In this parametrization the period of the annulus is $\tau \in \R$ and
$$
    X(z)=(G(z),\,y) \hbox{ with } X(z+\tau)=X(z)\,.
$$
We say that we have parameterized the surface by its {\it third component}.
We remark that $Q=\frac14 (dz)^2$ and $\o$ satisfies the $\sinh$-Gordon equation
\begin{equation}\label{eq:sinh-Gordon}
    \Delta  \o +  \sinh (\o) \cosh(\o) = 0\,.
\end{equation}
\begin{remark}
We relax the condition $\tau \in \R^\times$ to $\tau \in \C^\times$, but we will parameterize our annuli conformally such that $Q = \phi\,dz^2$ is constant, independent of $z$ and $4|\phi|=1$. This means that the third coordinate is linear.
\end{remark}
In summary we have proven the following
\begin{theorem}\label{immersionsinhgordon}
A proper minimal annulus is parabolic and $X: \C/ \tau \Z \to \SY^2 \times \R$ has conformal parametrization  $X(z)=(G(z),\,h(z))$ with
\begin{enumerate}
\item Harmonic map $G: \C/ \tau \Z \to \SY^2$, and $h(z) = {\rm Re} (-\mi e^{\mi\Theta/2} z)$.
\item Constant Hopf differential $Q= \frac{1}{4} \exp(\mi\Theta)\,dz^2$.
\item The metric of the immersion is $ds^2=\cosh^2(\o)\, dz\otimes d\bar z$.
\item The third coordinate of the unit normal vector is $n_3 =\tanh \o$.
\item The function $\o : \C/\tau \Z \to \R $ is a solution of \eqref{eq:sinh-Gordon}.
\end{enumerate}
\end{theorem}
\subsection*{Sym-Bobenko type formula.} We use a description of harmonic maps into the symmetric space $\bbS^2$ in terms of $\SU$-valued frames.

Identify $\su \cong \R^3,\,\bigl( \begin{smallmatrix} \mi w & u + \mi v \\ -u + \mi v & -\mi w \end{smallmatrix} \bigr) \cong (u,\,v,\,w)$, so $\SY^2 \subset \R^3$ consists of length $\|X\| = \sqrt{\det X} = 1$ elements in $\su$ and $\langle X_1,X_2 \rangle= -\frac12\mathrm{tr}\, \bigl( X_1X_2 \bigr)$. Pick
\[
    \sigma_3 = \bigl( \begin{smallmatrix} \mi & 0 \\ 0 & -\mi \end{smallmatrix} \bigr) \in \SY^2\,.
\]
Let $\mathbb{T}$ denote the stabiliser of $\sigma_3$ under the adjoint action of $\SU$ on $\su$, so that $\SY^2 \cong \SU / \mathbb{T}$. If $\pi :\SU \to \SU / \mathbb{T}$ is the coset projection, then a map $F:\Omega \subset \C \to \SU$ with $G (z) = \pi \circ F (z) $ is called a {\bf{frame}} of $G$, and we have $G (z) = F (z)  \sigma_3 F^{-1}(z) $.

Harmonic maps come in $\SY^1$-families (associated families), associate to the same solution of the sinh-Gordon equation. Here $\lambda\in \SY^1$ parameterizes an associated family of harmonic maps, and the frame of such an associated family is called an {\bf extended frame}.

A method to obtain an extended frame is to write down an integrable 1-form which integrates to an extended frame by solving the integrability condition.  We present a result which is similar to results in Bobenko \cite{Bob:cmc}. We next specify for a real solution of the sinh-Gordon equation \eqref{eq:sinh-Gordon} a minimal surface in $\SY^2 \times \R$ with a particularly simple vertical component.
\begin{theorem} \label{thm:sinh}
Let $\o:\C \to \R$ be a solution of the sinh-Gordon equation. Let $\delta, \gamma \in \SY^1$ be arbitrary but fixed,
and $F_\lambda (z)$ the solution of $F_\l^{-1}dF_\l  =\alpha (\o),\,F_\lambda(0) = \un$ where
\begin{equation} \label{eq:symb}
	\alpha (\o) = \frac{1}{4}\,
	\begin{pmatrix}
2  \o_z &  \mi \lambda^{-1} \delta  e^\o\\
 \mi  \gamma e^{-\o}&  -2\o_z
  \end{pmatrix}\, dz + \frac{1}{4}\,\begin{pmatrix}
-2 \o_{\bar{z}} &
 \mi \bar{\gamma} e^{-\o}\\
 \mi  \lambda \bar \delta \,e^{\o}&
2\o_{\bar{z}}
  \end{pmatrix} d \bar{z}\,.
\end{equation}
Then the map $X_\lambda (z) = ( F_\lambda (z)\, \sigma_3\, F_\lambda (z)^{-1},\, {\rm Re}( -\mi \sqrt{\gamma \delta \lambda^{-1}}z))$ with $\l \in \SY^1$ defines an associate family of conformal minimal immersions $X_\l :\C \to \SY^2 \times \R$ with metric
$$
	ds^2  = \cosh ^2 \o \,\,dz \otimes d\bar z\,,
$$
and Hopf differential
\[
	Q = \tfrac{1}{4}\,\gamma\,\delta\;\l^{-1}\,dz^2\,.
\]
\end{theorem}
\begin{remark}
\label{real}
{\bf Reality condition.}
For $\lambda \in \SY^1$, $\alpha_\lambda = \alpha(\omega)$ in \eqref{eq:symb} takes values in $\su$, so $F_\lambda$ takes values in $\SU$. For general $\lambda \in \C^\times$ we have $F_\l \in \SL (\C)$. From the relation $ \overline{ \a_{1/\bar \l}}^t = -\a_\l$, the solution of $F_\l^{-1}dF_\l  =\alpha (\o)$ satisfies the reality condition
\begin{equation} \label{eq:reality-F}
    \overline{F_{1/\bar \l}}^t = F_\l^{-1}\,.
\end{equation}
\end{remark}
\begin{proof} Decomposing $\alpha (\omega) = \alpha'\,dz +
\alpha''\,d\bar{z}$ into $(1,\,0)$ and $(0,\,1)$ parts, we
compute
\begin{equation*} \begin{split}
  &\alpha'_{\bar z} = \frac{1}{4}\,
  \begin{pmatrix}   2\o_{z\bar{z}} & \mi  \lambda^{-1}\o_{\bar{z}} \delta e^\o \\
    -\mi \gamma \o_{\bar{z}}e^{-\o} & - 2\o_{z\bar{z}} \end{pmatrix}\,,\\
  &\alpha''_z = \frac{1}{4}\,
  \begin{pmatrix} -  2\o_{z\bar{z}} & - \mi\bar{\gamma} \o_z  e^{-\o} \\
    \mi  \lambda \o_z \bar \delta e^\o &  2 \o_{z\bar{z}} \end{pmatrix}\,.
\end{split}
\end{equation*}
Using $F_\l^{-1} dF_\l = \a$ in the integrability condition $ (F_\l)_{z \bar z} = (F_\l)_{\bar z z}$ gives
$\alpha'_{\bar z}-\alpha''_z = [\a ' ,\, \a '']$, and a direct computation shows this is equivalent to the sinh-Gordon
equation \eqref{eq:sinh-Gordon}. Hence we can integrate $dF_\l = F_\l \,\alpha$ to obtain a map $F_\l:\C \to \SL (\C)$ and define  $G_\l= F_\l \, \sigma_3 \,F_\l^{-1}$. The vertical component is $h = \tfrac{\mi}{2} ( -(\gamma \delta)^{1/2}\lambda^{-1/2}z + (\gamma \delta)^{-1/2}\lambda^{1/2}\bar{z} $.
Its partial derivatives of are $ h_z = -\tfrac{\mi}{2}(\gamma\delta) ^{1/2}\lambda^{-1/2}$ and $ h_{\bar z} = \tfrac{\mi}{2} (\gamma \delta) ^{-1/2}\lambda^{1/2}$. Then
\begin{align*}
    \langle (X_\l)_z,\, (X_\l)_z \rangle &= -\tfrac{1}{2}\mathrm{tr}\, \bigl( [\alpha',\,\sigma_3]^2 \bigr)
        + ( h_z)^2 = 0\,,\\
    \langle  (X_\l)_{\bar z} ,\, (X_\l)_{\bar z}  \rangle &= -\tfrac{1}{2}\mathrm{tr}\, \bigl( [\alpha'',\,\sigma_3]^2 \bigr)
        + ( h_{\bar z})^2 = 0\,,
\end{align*}
and the conformal factor computes to
\begin{equation*} \begin{split}
 2 \langle (X_\l)_z ,\,(X_\l)_{\bar z} \rangle &= 2 \langle [\alpha',\,\sigma_3],\,[\alpha'',\,\sigma_3] \rangle  + 2\,( h_z)( h_{\bar z}) \\ &= \tfrac{1}{2} \cosh (2\omega) + \tfrac{1}{2} = \cosh^2 (\omega)\,.
\end{split}
\end{equation*}
As $Q=\langle (G_\l)_z,\, (G_\l)_z \rangle \,(dz)^2 = -\tfrac{1}{2}\mathrm{tr}\, \bigl([ \a',\,\sigma_3]^2 \bigr)\,(dz)^2 = 	\tfrac{1}{4}\gamma \delta \l^{-1}\,(dz)^2$ is holomorphic, we conclude that $G$ is harmonic.
\end{proof}
\begin{remark}
\label{IN1}
{\bf Isometric normalisation 1.}
By  conformal parametrization we can choose $4 |\phi|=1$ (the annulus is transverse to horizontal planes). We have constants $\delta, \gamma \in \SY^1$ which are related to the Hopf differential, namely $4 Q=\delta \gamma \l^{-1}\,dz^2$. We can normalize the parametrization with $\delta=1$ and a constant $|\gamma|=1$. For a given extended frame $F_\l$ which satisfies the equation $F_\l^{-1}dF_\l  =\alpha (\o),\,F_\lambda(0) = \un$, we consider the $U(1)$-gauge
\begin{equation} \label{eq:g-delta}
    g(\delta)=\begin{pmatrix} \delta^{1/2} &0 \cr 0 & \delta^{-1/2} \end{pmatrix} \in \bbT\,.
\end{equation}
Then $\tilde F_\l=g(\delta)^{-1} F_\l g(\delta)$ induces the immersion $\tilde X$ which differs from $X$ by a rotation in $\SY^2$.
The third coordinate $h=\tilde h$ is preserved while
$\tilde G_\l (z) =\tilde F_\l \sigma_3 \tilde F_\l ^{-1}=g(\delta)^{-1} G_\l (z) g(\delta)$
and for $\alpha(\omega)$ as in \eqref{eq:symb} with $\tilde \delta=1$ and $\tilde \gamma = \delta \gamma$ have
$$
\tilde F _\lambda ^{-1}d \tilde F_\lambda = g(\delta)^{-1} \alpha (\omega) g(\delta)\,.
$$
\end{remark}


\section{Polynomial Killing fields and spectral curves}
We explain in this section how solutions of the sinh-Gordon equation give rise to polynomial Killing fields as solutions of a Lax equation. Polynomial Killing fields in turn define spectral curves, which are hyperelliptic Riemann surfaces.

If $\o$ is a solution of the sinh-Gordon equation, we consider a deformation $\o_t = \o + t\,u + O(t^2)$. If $\o_t$ is a one parameter family of solutions of the sinh-Gordon equation, then the variational function $u:\C \to \R$ satisfies the {\bf linearized sinh-Gordon equation}
\begin{equation}
\label{eq:lsg}
    \Delta u + u\,\cosh (2\o)  =0\,.
\end{equation}
\begin{definition}
A solution $\o$ of the sinh-Gordon equation is of {\bf finite type} if there exist $g \in \N$ such that
\begin{equation} \label{potentiel}
  \Phi_\l (z) = \frac{\l^{-1}}{4} \begin{pmatrix}   0 &  \mi e^{\o}   \cr 0 & 0   \end{pmatrix}+ \sum_{n=0}^g \l^n \begin{pmatrix}   u_n (z)  & e^{\o} \tau_n (z)  \cr e^{\o} \sigma_n (z)  & -u_n (z)  \end{pmatrix}
\end{equation}
is a solution of the Lax equation
\begin{equation} \label{eq:lax-eq}
d \Phi_\l = [\Phi_\l,\,\a (\o)]
\end{equation}
for some functions $u_n, \tau_n, \sigma_n :  \C \to \C$, and some $\gamma \in \bbS^1$, and $\delta =1$ in $\alpha(\omega)$.
\end{definition}
\begin{proposition} \label{th:PinS-iteration}
Suppose $\Phi_\l$ is of the form \eqref{potentiel} for some arbitrary $\o:\C \to \R$, and that $\Phi_\l$ solves the Lax equation \eqref{eq:lax-eq} with $\alpha (\o)$ as in \eqref{eq:symb}, $\delta=1$ and $|\gamma |=1$. Then:
\begin{enumerate}
\item The function  $\o$ is a solution of the sinh-Gordon equation \eqref{eq:sinh-Gordon}.
\item The functions $u_n$ are solutions of the linearized sinh-Gordon equation \eqref{eq:lsg}.
\item The following iteration gives a formal solution of $d \Phi _\l = [\Phi _\l,\,\a(\o)]$.  \\ Let $u_n, \,\sigma_n, \,\tau_{n-1}$, with $u_n$ solution of \eqref{eq:lsg} be given. Now solve the system
\[
    \tau_{n;\bar z}= \tfrac{1}{2}\,\mi \bar \gamma\, e^{-2\omega} u_n\,,  \qquad
    \tau_{n; z}= 4\mi \bar\gamma\,\omega_z u_{n;z} - 2 \mi \bar\gamma\, u_{n;zz}
\]
for $\tau_{n; z}$ and $\tau_{n;\bar z}$. Then define $u_{n+1}$ and $\sigma_{n+1}$ by
\[
    u_{n+1} = -2\mi \tau_{n;z}-4\mi \omega_z \tau_n\,, \qquad
    \sigma_{n+1} = \gamma\,e^{2\omega} \tau_n + 4\mi\gamma\, u_{n+1; \bar z}\,.
\]
\item Each $\tau_n$ is defined up to a complex constant $c_n$, so $u_{n+1}$ is defined up to $-4\mi c_n \o_z$.
\item $u_0=\omega _z,  u_{g-1}= c\,\omega_{\bar z}$ for some $c \in \C$, and $\l^{g} \overline{ \Phi _{1/ \bar \l}} ^t$ also solves \eqref{eq:lax-eq}.
\end{enumerate}
\end{proposition}
\begin{proof}
Inserting \eqref{potentiel} into \eqref{eq:lax-eq} and comparing coefficients yields
\begin{subequations} \label{eq:pinkall-sterling}
\begin{gather}
\label{eq:psA}
	4 u_{n;z} + \mi e^{2\o}\sigma_{n+1} - \mi\gamma\tau_n =0 \,,\\
\label{eq:psB}
	4 u_{n; \bar z}+ \mi\bar \gamma \sigma_n - \mi e^{2\o}\tau_{n-1} =0\,,\\
\label{eq:psC}
	4 \o_z  \tau_n + 2 \tau_{n;z} - \mi u_{n+1} =0\,,\\
\label{eq:psD}
	2e^{\o} \tau_{n;\bar z}- \mi \bar \gamma  e^{-\o} u_n =0\,,\\
\label{eq:psE}
	2e^{\o} \sigma_{n;z}+ \mi  \gamma  e^{-\o} u_n =0\,,\\
\label{eq:psF}
	4 \o_{\bar z}\,\sigma_n + 2 \sigma_{n, \bar z} + \mi  u_{n-1} =0\,.
\end{gather}
\end{subequations}
\noindent (1) Solving \eqref{eq:psB} for $\sigma_{n+1}$, \eqref{eq:psC} for $u_{n+1}$, \eqref{eq:psD} for $\tau_{n;\bar z}$, and inserting these, and $u_{n+1;\bar z}$ and $\tau_{n;\bar z z}$ into \eqref{eq:psA} gives $e^{2\o} \mi \gamma \tau_n (16 \o_{z \bar z} - e^{-2\o} + e^{2\o}) = 0$, which implies \eqref{eq:sinh-Gordon} if $\tau_n \neq 0$.

\noindent (2) $\bar{\partial}\eqref{eq:psA} - \tfrac{1}{2} \mi e^{2\o} \eqref{eq:psF} + \tfrac{1}{2}\mi \gamma e^{-\o}\eqref{eq:psD}$ reads
$$
    4u_{n;z \bar z} + \tfrac{1}{2} u_n \left( e^{2\o} + \gamma \bar \gamma e^{-2\o} \right) = 0\,.
$$
\noindent (3) The equation for $\tau_{n;\bar z}$ is \eqref{eq:psD}. Taking the $z$-derivative of \eqref{eq:psA} and using \eqref{eq:psA}, \eqref{eq:psC} and \eqref{eq:psE} gives $\tau_{n;z}= 4\mi \bar\gamma\o_z u_{n;z} - 2\mi \bar\gamma u_{n;zz}$. The equations for $u_{n+1},\,\sigma_{n+1}$ are given by \eqref{eq:psC} respectively \eqref{eq:psB}.

\noindent (4) In the iteration (3) the function $\tau_n$ is determined up to an integration constant. This gives an additional term  $\o_z$ in $u_{n+1}$.

\noindent (5) Left to the reader.
\end{proof}
Pinkall-Sterling \cite{PinS} constructed a series of special solutions of the induction of Proposition \ref{th:PinS-iteration} (3) via an auxiliary function $\phi$ as follows: For a given solution $u_n$ of the linearized sinh-Gordon equation \eqref{eq:lsg}, consider the function $\phi : \C \to \C$ defined by
\begin{equation}\label{eq:iteration-phi}
    \phi _{n;z} =  4 \o_z u_{n;z} \,, \qquad
    \phi_{n;\bar z}=   -u_{n} \sinh \o \cosh \o\,.
\end{equation}
Then $\tau_n = 2\mi \bar\gamma\,\left( \tfrac{1}{2}\phi_n - u_{n;z} \right)$ and $u_{n+1}:=(u_{n})_{zz}- \o_z \phi_{n}$.
This defines a hierarchy of solutions of \eqref{eq:lsg}.
Applying this iteration to the trivial solution $u_{0} \equiv 0$ yields the sequence,
whose first four terms are
\begin{equation*}\begin{split}
&u_{0}=0\,,\\
&u_{1}= \o_z \,,\\
&u_{2}= \o_{zzz}-2\o_z^3\,,\\
&u_{3}= \o_{zzzzz}-10 \o_{zzz}\o_{z}^3-10 \o^2_{zz} \o_{z} + 6\o_{z}^5 \,.
\end{split}
\end{equation*}
This infinite sequence produces solutions of the linearized sinh-Gordon equation on $\C$. These come from the iteration \eqref{eq:iteration-phi}, and Pinkall-Sterling prove in Proposition~3.1 \cite{PinS}, that $\phi_n$ depends only on $\o$ and its $k$-th derivatives with $k \leq 2n+1$. The fact that we consider on $\C$ a uniformly bounded solution of the sinh-Gordon equation $\o : \C \to \R$, implies by Schauder estimates that each $u_n$ is uniformly bounded on $\C$.
\begin{proposition}
A proper minimal annulus $A$ immersed in $\SY^2 \times \R$ with bounded curvature and linear area growth has a metric $\rmd s^2=\cosh ^2 \o |\rmd s|^2$, where $\o:\C/\tau\C \to \R$ is a finite type solution of the sinh-Gordon equation. We say that the annulus is of finite type.
\end{proposition}
\begin{proof} A first step is to prove that the function $\o:\C/\tau\C \to \R$ is uniformly bounded. Consider a sequence of points $p_n$ in $A$ such that $\o (p_n)$ is diverging to infinity and consider a sequence of translations $t_n e_3$ such that $A + t_n e_3$ is a sequence of annuli with $p_n + t_n e_3$ points of $\SY^2 \times \{0\}$. Then by the bounded curvature hypothesis there is a sub-sequence converging locally to $A_0$, a properly immersed minimal surface. The linear area growth assumption assure that $A_0$ is an annulus. But our hypothesis leads to a pole occurring at the height $t =0$ since $|\o| \to \infty$. The limit normal vector $n_3 (p_n) =\tanh \o_n (p_n) \to \pm 1$ and the annulus $A _0$ would be tangent to the height $\SY^2 \times \{0\}$, a contradiction to the maximum principle. Thus
$$\sup _{z\in A} |\o| \leq C_0\,.$$
Now we apply Schauder estimates to the sinh-Gordon equation to obtain a ${\mathcal C}^{k,\a}$ estimate on the solution of the sinh-Gordon equation on $\C / \tau \bbZ$. There exists a constant $C_0>0$ such that for any $k \in \N$
$$|\o|_{A, k,\alpha} \leq C_0\,.$$
Meeks-Perez-Ros \cite{MPR} provide us with the following
\begin{theorem}
\label{finitekernel}\cite{MPR}
An elliptic operator $L u= \Delta u + q\,u$ on a cylinder $\SY^1 \times \R$ has for bounded and continuous $q$ a finite dimensional kernel on the space of uniformly bounded ${\mathcal C}^2$ -functions on $\SY^1 \times \R$.
\end{theorem}
Since solutions $u_0,\,u_1,\,u_2...$ of the linearized sinh-Gordon equation are solutions depending only on $\o$ and its higher derivatives, this family is a finite dimensional family by Theorem~\ref{finitekernel}. Thus there is a $g \in \N$, and there exist $a_i \in\C$ such that
$$
	\sum _{i=1}^g   a_i u_{i}=0\,.
$$
This algebraic relation implies that $\o$ is of finite type, and ensures the existence of a polynomial Killing field $\Phi_\l$ of degree $g$. To achieve that, one has to prescribe the right constants $c_0,\,c_1,\,\ldots,\,c_g$ in the iteration procedure, and set the $g+1$-coefficient to zero.
\end{proof}
%
%
\subsection*{Potentials and polynomial Killing fields.} To parameterize real solutions of the sinh-Gordon equation we make the following
\begin{definition}\label{def pot}
The set of potentials is
\begin{align*}
    {\mathcal P}_g= \{ \,\xi_\l=\sum_{d=-1}^{g}\hat \xi_d\lambda^d  \mid  \hat  \xi_{-1}\in  \bigl( \begin{smallmatrix} 0 &  \mi\R^+ \cr 0 & 0 \end{smallmatrix} \bigr)  , {\rm tr}(\hat \xi_{-1}\hat  \xi_0) \neq 0,\, \\
    \hat \xi_d=-\overline {\hat \xi_{g-1-d}}^t\in\Sl(\C)\mbox{ for }d=-1,\ldots,g \, \}.
\end{align*}
\end{definition}
\begin{remark}
\label{IN2} {\bf Isometric normalization 2.} For $\delta \in \SY^1$, we denote by $\pk (\delta)$, potentials with residues
$$
    \hat  \xi_{-1}  \in  \bigl( \begin{smallmatrix} 0 &  \mi \delta\R^+ \cr 0 & 0 \end{smallmatrix} \bigr)\,.
$$
These correspond to the normalization of remark \ref{real} and there is an isomorphism $\pk (\delta) \to \pk$, given by
$$
    \xi_\l \mapsto g(\delta)^{-1} \xi_\l g(\delta)\,.
$$
\end{remark}
Each $\xi_\l \in\pk$ satisfies the {\bf reality condition}
\begin{equation} \label{realcoeff}
	\l^{g-1} \overline{ \xi _{1/\bar \l}}^t = - \xi _\l\,.
\end{equation}
In other words, for $\xi_\l \in \pk$, we have a map $\SY^1 \to \su, \lambda \mapsto \lambda^{\frac{1-g}{2}} \xi_\lambda$. The polynomial
\begin{equation*}
    a(\l) : = - \l \det \xi _\l
\end{equation*}
then satisfies the reality condition
\begin{equation} \label{eq:real}
	\l^{2g}\overline{ a (1/\bar \l)} = a (\l)\,.
\end{equation}
On $\su$ the determinant is the square of a norm, thus we have for $\l \in \SY^1$ that
\begin{equation}\label{eq:hermit}
 	\l^{-g}\,a(\l)  \leq 0 \hbox{ for } \l \in \SY^1\,.
\end{equation}
When $g$ is even, $\hat \xi_0,...,\hat \xi_{\frac{g}{2}-1}$ are independent $2 \times 2$ traceless complex matrices. For odd $g$, $\hat \xi_0,...,\hat \xi_{\frac{g-3}{2}}$ are independent $2 \times 2$ traceless complex matrices and
$\hat \xi_{\frac{g-1}{2}} \in \su$. Thus the space of potentials ${\mathcal{P}}_g$ of real finite type solutions of the $\sinh$-Gordon equation is an open subset of a $3g+1$ dimensional real vector space.  The condition ${\rm tr}(\hat \xi_{-1} \hat \xi_0) \neq 0$ implies that $a(0)\neq 0$ and by symmetry the highest coefficient of $a$ is therefore non-zero. Thus $\lambda \mapsto a(\lambda)$ is a polynomial of degree $2g$ with complex coefficients, and we denote such by $\C^{2g} [\l]$. Define
\begin{equation} \label{eq:calM-sets} \begin{split}
    {\mathcal{M}}_g &=\{ a \in \C^{2g} [\l] \mid a(\l)= -\l \det \xi _\l \hbox{ with } \xi_\l \in {\mathcal{P}}_g \} \\
                                  &  = \{ a \in \C^{2g} [\l] \mid  a(0)\neq 0,   \l^{2g}\overline{ a (1/\bar \l)} = a (\l) \\
                                    &\qquad\qquad \hbox{ and }\l^{-g}\,a(\l)  \leq 0 \hbox{ for } \l \in \SY^1  \}\,, \\
    {\mathcal{M}}_g^0 & =\{ a \in {\mathcal{M}}_g  \mid  \l^{-g} a(\l) < 0 \hbox{ for } |\l|=1 \}\,.
\end{split}
\end{equation}
Thus ${\mathcal{M}}_g^0$ is an open subset of the $2g+1$ dimensional real vector space ${\mathcal{M}}_g$.
\begin{definition}\label{def pol}
Polynomial Killing fields are maps $\zeta_\l : \C \to \pk$
(see Definition~\ref{def pot}) which solve the Lax equation
$$
    d\zeta _\l = [\,\zeta_\l,\,\alpha(\omega)\,] \qquad\mbox{ with }\quad
    \zeta_\l(0)=\xi_\l \in \pk.
$$
\end{definition}
We use solutions $\Phi_\l$ of the Lax equation to construct polynomial Killing related to a finite type solution $\o : \C \to \R$ of the sinh-Gordon equation.
\begin{lemma}
For a solution $\Phi_\l$ of \eqref{eq:lax-eq} with $\delta =1$ in $\a (\o)$ (see formula (\ref{eq:symb})), there exists constants $\gamma \in \SY^1$ and
$k \in \R^+$  such that
$$
    \zeta_\l (z)=k \Phi_\l (z) - k \l^{g-1}\; \overline{ \Phi _{1/ \bar \l} (z)}^{t} \hbox{ and } \zeta_\l (0)=\xi_\l
$$
is a polynomial Killing field.
\end{lemma}
\begin{proof}
The map $\zeta_\l$ satisfies the reality condition (\ref{realcoeff}) so that
$$\l^{g-1} \overline{ \zeta _{1/\bar \l} (z)}^t = - \zeta _\l (z).$$
It remains to prove that
the residues $\hat \zeta_{-1}$ and $\hat \xi_{-1}$ are upper triangular
with purely imaginary non-zero coefficient and ${\rm trace}(\hat \xi_{-1}\hat  \xi_0) \neq 0$. We use the following
\begin{remark}
\label{IN3}{\bf Isometric normalization 3.} We write $\a_{\l, \delta, \gamma} (\o) := \a (\o)$ for the $1$-form \eqref{eq:symb} and $\Phi _{\l, \delta, \gamma}(z) :=\Phi _{\l} (z)$ the associate solution \eqref{potentiel} of the Lax equation \eqref{eq:lax-eq}. We use the unitary matrix $g(\delta)$ defined in \eqref{eq:g-delta}. Now $\Phi _{\l \delta, 1, \gamma}$ solves \eqref{eq:lax-eq} with $\a_{\l\delta , 1, \gamma} (\o)=\a_{\l, \delta^{-1}, \gamma} (\o)$, so $g(\delta) \Phi _{\l \delta, 1, \gamma} g(\delta)^{-1}$ solves \eqref{eq:lax-eq} with $\a (\o)=
g(\delta)  \a_{\l, \delta^{-1}, \gamma}g(\delta)^{-1}= \a_{\l, 1, \delta^{-1}\gamma}$. We conclude that
$$
    \Phi _{\l , 1, \delta^{-1}\gamma}=g(\delta) \Phi _{\l \delta, 1, \gamma} g(\delta)^{-1}\,.
$$
In particular, if $\xi_\l \in \pk (\delta)$, then a solution of the Lax equation of definition \ref{def pol} satisfies $\zeta_\l : \C \to  \pk (\delta)$.
This can be deduced from proposition \ref{th:PinS-iteration}, where $\tau_{-1}(z)=\tau_{-1}(0)$.
\end{remark}
The remark \ref{IN3} proves that changing $\gamma$ by $\delta^{-1} \gamma$ changes the highest coefficient of $\Phi_\l$ by
$$
    \sigma_g (\delta^{-1} \gamma)=\delta^{g-1}\sigma_g (\gamma)\,.
$$
We compute the residue $\hat \xi_{-1}$ at $z=0$ with $\tilde \gamma= \delta ^{-1} \gamma$, and choose a unimodular number $\delta$, such that
$$
    \tfrac{1}{4}\mi e^{\o} - \bar \sigma_g (\tilde \gamma) e^{\o}= \tfrac{1}{4} \mi e^{\o} - \delta^{1-g}\bar \sigma_g (\gamma) e^{\o} \in \mi \R^+\,.
$$
We choose $k$ to normalize the residue with $k ^{-1} =1+4 \mi\delta^{1-g}\bar \sigma_g (\gamma) \in \R^+$.
There remains to compute with $\sigma_0= \mi \gamma e^{-2\o}/4$,
$$
    {\rm tr}(\hat \xi_{-1}\hat  \xi_0)=\left( \tfrac{1}{4}\mi e^{\o}\right) \left( \tfrac{1}{4}\mi \gamma e^{-\o}\right)=-\gamma /16 \neq 0\,.
$$
Since $\xi_\l \in \pk$, the Lax equation assures that $\zeta_\l (z) \in \pk$ (see remark \ref{laxomega}).
\end{proof}
\subsection*{Spectral curve.} Suppose $\xi_\l \in \pk$ and $\zeta_\l$ is the polynomial Killing field with $\zeta_\l(0) = \xi_\l$. Suppose further that the polynomial $a(\lambda) = -\lambda \det \xi_\l$ has $2g$-pairwise distinct roots. Define
\begin{equation} \label{eq:sigma} \begin{split}
  \Sigma^* &= \{(\nu,\l) \in \C^2 \mid \det(\nu\; \un - \zeta_\l)=0\} \\
    &= \{ (\nu,\,\l) \in   \C^2 \mid \nu ^2 =-\det \xi _\l=\l^{-1}a(\l)\}\,.
    \end{split}
\end{equation}
By construction we have a map $\lambda:\Sigma^* \to \C^\times$ of degree 2, which is branched at the $2g$ simple roots of the polynomial $a$. By declaring the points over $\lambda = 0,\,\infty$ to be two further branch points, we then have $2g+2$ branch points. This 2-point compactification $\Sigma$ is called the {\bf{spectral curve}} of the polynomial Killing field $\zeta_\l$.

The Riemann-Hurwitz formula gives that the spectral curve $\Sigma$ is a hyperelliptic Riemann surface of genus $g$, and its genus is called the {\bf spectral genus}. It has three involutions
\begin{equation}\label{eq:involutions} \begin{split}
\sigma&:(\lambda,\nu)\mapsto(\lambda,-\nu) \,,\\
\varrho&:(\lambda,\nu)\mapsto(\bar{\lambda}^{-1},- \bar{\lambda}^{1-g}\bar{\nu})\,,\\
\eta&:(\lambda,\nu)\mapsto(\bar{\lambda}^{-1},\bar{\lambda}^{1-g}\bar{\nu})\,.
\end{split}
\end{equation}
The involution $\sigma$ is called the hyperelliptic involution. Note that $\eta$ has no fixed points ($a(1) \in \R^-$) and $\varrho$ fixes $\SY^1$ pointwise. In particular, roots of $a$ are symmetric with respect to inversion across the unit circle so that $a(\alpha_i) = 0 \Leftrightarrow  a(1/ \bar{\alpha}_i) = 0$.


\section{Harmonic maps and Weierstrass representation}
The generalized Weierstrass representation \cite{DorPW} gives a correspondence between harmonic maps, extended frames and potentials. To formulate the generalized Weierstrass representation for harmonic maps into $\SY^2$ we need various loop groups and a loop group factorization.

For real $r \in (0,1]$, denote the circle $ \SY_r=\{ \l \in \C \mid |\l|=r \}$, the disk $I_r=\{ \l \in \C \mid |\l| <r \}$ and the annulus $A_r = \{ \l \in \C \mid r < | \l | < 1/r \}$. The loop group of $\SL(\C)$ is the infinite dimensional Lie group $\Lambda_r \SL(\C)={\mathcal{O}}(\SY_r,\SL (\C))$ of analytic maps $\SY_r \to \SL (\C)$.

We need two subgroups. The first is
$$
\Lambda _r\SU  = \{ F_\l  \in \mathcal{O}(A _r,\, \SL ( \C) \,\mid\,  F_{\l \in  \SY^1}\in \SU \}\,.
$$
Thus
\begin{equation}
  F_\l  \in \Lambda_r\SU \Longleftrightarrow \overline{F_{1/\bar{\lambda}}}^t = F^{-1}_\l\,.
\end{equation}
The second subgroup that participates in the Iwasawa decomposition is
\begin{equation*}
\begin{split}
\Lambda _r^+ \SL (\C) &=\{  B_\l \in  {\mathcal{O}}(I_r \cup \SY_r,\SL (\C)) \mid
B_0= \bigl( \begin{smallmatrix} \rho & c  \cr 0 & 1/\rho \end{smallmatrix} \bigr) \\
    &\qquad \qquad \hbox{ for } \rho \in \R^+ \hbox{ and } c\in \C \}\,.
\end{split}
\end{equation*}
The normalization that $B_0$ is upper-triangular with real diagonals ensures that
$$
\Lambda_r \SU \cap \Lambda_r^+ \SL(\C) = \{ \mathbbm{1} \}\,.
$$
The following important result is due to Pressley-Segal \cite{PreS}, and generalized by McIntosh \cite{McI}.
\begin{theorem}
Multiplication $\Lambda_r \SU  \times \Lambda _r^+\SL (\C) \to \Lambda_r \SL (\C)$ is a real analytic diffeomorphism onto. The unique splitting of $\phi_\l \in \Lambda_r \SL (\C)$ into $\phi _\l= F_\l\, B_\l$ with $F_\l \in \Lambda_r \SU$ and $B _\l \in  \Lambda _r^+ \SL (\C)$ is called the {\bf r-Iwasawa decomposition} of $\phi_\l$ or just {\bf Iwasawa decomposition} when
$r=1$.
\end{theorem}
Before specializing to the finite type theory of harmonic maps $G:\C \to \bbS^2$, let us briefly recall the generalized Weierstrass representation \cite{DorPW}. Set
$$
\Lambda_{-1}^\infty \mathfrak{sl}_2 (\C)  = \{ \xi_\l \in {\mathcal{O}}(\C^\times ,\,\mathfrak{sl}_2 (\C)) \,\mid \, \left.(\lambda \xi_\l)\right|_{\lambda =0} \in \bigl( \begin{smallmatrix} 0 & \C^\times  \\ 0 & 0 \end{smallmatrix} \bigr) \}\,.
$$
A {\bf potential} is a holomorphic 1-form on $\C$ with values in $\Lambda_{-1}^\infty \mathfrak{sl}_2 (\C)$.
Suppose that we have such a potential $\xi_\l\,dz$ with $\xi_\l \in \Lambda_{-1}^\infty \mathfrak{sl}_2 (\C)$. To obtain a corresponding extended frame $F_\l$ is a two step procedure:
\begin{enumerate}
\item Solve the holomorphic ODE $d\phi_\l = \phi_\l \xi_\l$ to obtain a map $\C \to \Lambda_r \SL (\C)$, $z \mapsto \phi_\l (z)$.
\item The r-Iwasawa factorization $\phi _\l (z) = F_\l (z) \, B_\l (z)$ at each $z \in \C$ gives an extended frame $\C \to \Lambda_r \SU, z \mapsto F_\l (z)$.
\end{enumerate}
Note that while $\phi_\l$ is holomorphic in $z \in \C$, the resulting extended frame is not, since it also depends on $\bar{z}$ by the reality condition \eqref{eq:reality-F}. It is proven in \cite{DorPW} that each extended frame can be obtained from a potential $\xi_\l$ by the Iwasawa decomposition. Hence for any conformal minimal immersion $X = (G,\,h):\C \to \SY^2 \times \R$ there is a potential and corresponding extended frame which frames $G$.

An extended frame is of (semi-simple) {\bf{finite type}}, if there exists a $g \in \bbN$, and it has a corresponding potential $\xi_\l dz$ with $\xi_\l \in \pk \subset \Lambda_{-1}^\infty \mathfrak{sl}_2 (\C) $. Hence harmonic maps of finite type come from constant (1,\,0)-forms with values in the finite dimensional space $\pk$, and thus have an algebraic description. In the finite type case the first step in the above two step procedure is explicit, since then $\phi_\l = \exp(z\,\xi_\l)$. Thus extended frames of finite type are obtained by factorizing $\exp(z\,\xi_\l) = F_\l\,B_\l$ with $\xi_\l \in \pk$. This step can be made explicit in terms of theta functions on the spectral curve (see Bobenko \cite{Bob:tor}).

Expanding a polynomial Killing field $\zeta_\l: \C \to \pk$ as
\begin{equation} \label{eq:zeta-expansion}
    \zeta_\l (z) = \bigl( \begin{smallmatrix} 0 & \beta_{-1}(z)  \\  0 & 0 \end{smallmatrix} \bigr) \lambda^{-1} + \bigl( \begin{smallmatrix} \alpha_0 (z)
     & \beta_0  (z) \\ \gamma_0 (z) & -\alpha_0 (z)  \end{smallmatrix} \bigr) + \ldots + \bigl( \begin{smallmatrix} \alpha_g (z)  & \beta_g (z)  \\ \gamma_g (z)  & -\alpha_g (z)  \end{smallmatrix}\bigr) \lambda^g
\end{equation}
we associate a matrix 1-form defined by
\begin{equation} \label{eq:alpha-zeta}
    \a(\zeta_\l)=\begin{pmatrix}
    \a_0 (z)  & \beta_{-1}(z) \l^{-1}  \cr
    \g_0 (z)  & -\a_0 (z)
    \end{pmatrix} {\rm d}z - \begin{pmatrix}
    \overline{\a_0 (z)}  & \overline{\g_0 (z)}   \cr
    \overline{\beta_{-1} (z)}\,\l & - \overline{\a_0 (z)}
    \end{pmatrix} {\rm d} \bar z\,.
\end{equation}
The following Proposition is well known, and the correspondence between potentials, polynomial Killing fields and extended frames is known as 'Symes method' \cite{BurP_adl, BurP:dre, Symes_80}.
\begin{proposition}
For each $\xi_\l\in\pk$ there is a unique polynomial Killing field $\zeta_\l:\C\to\pk$ solving
$d\zeta_\l =[\,\zeta_\l ,\,\alpha (\zeta_\l)\,]$ with $\zeta_\l( 0)=\xi_\l$. The unitary factor $F_\l :  \C \to \Lambda \SU$ of the Iwasawa decomposition $\exp ( z\,\xi_\l) = F_\l \,B_\l$ is a solution of $F_\l^{-1} dF_\l =\a (\zeta_\l)$ with initial value $F_\l (0) = \un$ and $\zeta_\l (z)  = B_\l (z)  \,\xi_\l \,B_\l^{-1} (z)  = F_\l^{-1} (z) \,\xi_\l \,F_\l (z) $.
\end{proposition}
\begin{proof}
Clearly $\zeta _\l= B_\l (z)  \,\xi _\l  \,B_\l^{-1} (z)  = F_\l^{-1}(z) \,\xi_\l  \,F_\l (z)$ uniquely solves $d\zeta_\l =[\,\zeta_\l  ,\,F_\l^{-1} dF_\l\,]$ with $\zeta_\l (0)=\xi_\l$, so it remains to show $F_\l^{-1}dF_\l=\alpha (\zeta_\l)$.

Now $F^{-1}_\l dF_\l = B_\l\xi_\l B_\l^{-1} - dB_\l B_\l^{-1}$, so by the reality condition \eqref{eq:reality-F} we have $F_\l^{-1} dF_\l= \mathfrak{a}_{-1} \l^{-1} + \mathfrak{a}_0 + \mathfrak{a}_1 \l$. If we decompose  $\mathfrak{a}_j = \mathfrak{a}_j'dz + \mathfrak{a}_j''d\bar z$ into (1,\,0) and (0,\,1) parts, then we have in addition that
\[
    \mathfrak{a}_0'' = -\overline{\mathfrak{a}'_0}^{\,t}\,,\quad \mathfrak{a}_1'' = -\overline{\mathfrak{a}'_{-1}}^{\,t}\,,\quad
    \mathfrak{a}_{-1}'' = -\overline{\mathfrak{a}'_1}^{\,t}\,.
\]
Now $F_\l^{-1} (F_\l)_z = \zeta _\l  - (B _\l)_z B_\l^{-1}$, and $( B_\l)_z \,B_\l^{-1}$ is holomorphic at $\l =0$.
Hence $\mathfrak{a}_{-1}' = \hat \zeta_{-1}$. Further, $F_\l^{-1} (F_\l)_{\bar z} = - ( B_\l)_{\bar z} B_\l^{-1}$ implies $\mathfrak{a}_{-1}'' = 0$. It remains to determine $\mathfrak{a}_0$.

Expand $B_\l = \hat B_0 + \hat B_1 \lambda + \ldots$. Now $\mathfrak{a}_0' = \zeta_0 -(\hat B_0)_z\,\hat B_0^{-1}$. Since $\hat B_0$ is upper-triangular, then so is $(\hat  B_0)_z\,\hat B_0^{-1}$. Hence the lower-diagonal term of $\mathfrak{a}_0'$ is $\gamma_0$, and the upper-diagonal term of $\mathfrak{a}_0''$ is $-\bar\gamma_0$. Also $\mathfrak{a}_0'' = - (\hat B_0)_{\bar z}\,\hat B_0^{-1}$ is upper-triangular, so the lower-diagonal entry of $\mathfrak{a}_0''$ is zero, and consequently also the upper-diagonal entry of $\mathfrak{a}_0'$ is zero.

Finally, writing $B(0)=\hat B_0 = \bigl( \begin{smallmatrix} \rho & c \\ 0 & 1/\rho \end{smallmatrix} \bigr)$, and
$\mathfrak{a}'_0 = \bigl( \begin{smallmatrix} u & v \\ w & -u \end{smallmatrix} \bigr)$, then
$\bar u =\rho^{-1}\rho_{\bar z}$ and $u = \alpha_0 - \rho^{-1}\rho_z$. These two equations, and since $\rho$ is real analytic, give $2u = \alpha_0$.
\end{proof}
\begin{remark}
\label{laxomega}
With initial data $\beta_{-1}(0) \in \mi \R^+$, the Lax equation gives $\beta_{-1}(z)  \in \mi  \R^+$. For a given potential $\xi_\l$ and polynomial Killing field $\zeta_\l :\C \to \pk $ we define $\o: \C \to \R$ by setting $4 \beta_{-1} :=\mi e^\o$.
The iteration implies that $2 \alpha_0 = \o_z$. To express $\gamma_0$ in terms of $\alpha_0$ and $\beta_{-1}$, we consider the Lax equation and find $(\gamma_0)_{\bar{z}} = -2 \bar \alpha_0 \gamma_0$. Then $\gamma_0 = q e^{-\o}$ where $q$ is a holomorphic function. The term $q$ is constant. The reason is that along the parameter $z$, we have $a(\l)=-\l \det \zeta_\l (z)=-\l \det \xi_\l$ and $a(0)=\beta_{-1}\gamma_0=q/4$. Coefficients of $\zeta_\l$ depend only on higher derivatives of $\o$ point wise in $z$, and $\a(\zeta_\l)=\a (\o)$.
\end{remark}

\begin{remark}{\bf Isometric normalization 4.} \label{IN4}
With initial data $\beta_{-1}(0) \in \mi \delta \R^+$ (i.e. $\xi_\l \in \pk (\delta)$), the Iwasawa decomposition
gives a solution of the Lax equation $\zeta_\l : \C \to \pk (\delta)$ given by $\zeta_\l (z)=F_\l ^{-1}(z) \xi_\l  F_\l (z)$.
\end{remark}


\section{Isospectral set}
The set $I(a)$ consists of all initial conditions $\xi_\l$ which give rise to the same spectral curve $\Sigma$ and the same off-diagonal product $a(0)= \beta_{-1} \gamma_0$.

\begin{definition}
Define for polynomial Killing field $\zeta_\l:\C \to\pk (\delta)$ as in \eqref{eq:zeta-expansion}, and $a(\l) = -\lambda \det \zeta_\l$ the set
\begin{equation*} \begin{split}
    I_{\delta}(a):= \{ \,\xi _\l \in &\pk  (\delta)\mid \l \det \xi_\l = -a(\l) \hbox{ and } \\ &\beta_{-1} \gamma_0 = a(0)= -\tfrac{1}{16} e^{\mi (1-g)\theta} := -\tfrac{1}{16} e^{\mi\Theta} \,\}\,.
\end{split}
\end{equation*}
When $\delta=1$, we write $I(a)$. The set $I(a)$ is called the {\bf isospectral set} of the polynomial Killing field $\zeta_\l$.
\end{definition}
We next define the {\bf isospectral action} $\pi: \mathbb{C}^g \times I(a) \to I(a)$.
\begin{definition} \label{groupaction}
Let $\xi_\l \in I(a)$ and $t=(t_0,\ldots,t_{g-1}) \in \C^g$, and
\begin{equation} \label{eq:isospectral action}
\exp\,\bigl( \,\xi_\l \sum_{i=0}^{g-1}\lambda^{-i} t_i \, \bigr) = F_\l(t)B_\l(t)
\end{equation}
the Iwasawa factorization. Define the map $\pi(t) : I(a) \to I(a)$ by
\begin{equation*}
	\pi(t) \,\xi_\l =B_\l(t) \, \xi_\l\, B_\l^{-1}(t)\,.
\end{equation*}
\end{definition}
Since $F_\l(t)B_\l(t)$ commutes with $\xi_\l$ we have $$\pi(t)\,\xi_\l =B_\l(t) \, \xi_\l\, B_\l^{-1}(t) = F_\l^{-1}(t) \,\xi_\l\, F_\l(t)\,.$$
\begin{proposition}
The map $\pi(t) : I(a) \to I(a)$  defines a commutative group action
\begin{equation*}
	\pi(t + t') = \pi(t)\,\pi(t') = \pi(t')\,\pi(t)\,.
\end{equation*}
\end{proposition}
\begin{proof}
For $t,\,t'\in\mathbb{C}^g$ we have $$F_\l (t+t')B_\l (t+t')=F_\l (t)B_\l (t)F_\l (t')B_\l (t') = F_\l (t')B_\l (t')F_\l (t)B_\l (t)\,.$$ Hence
\begin{equation*} \begin{split}
	\exp ( B_\l (t) \xi_\l B_\l ^{-1}(t) \,\Sigma \lambda^{-i}t'_i ) &= B_\l(t) F_\l(t') B_\l(t') B_\l^{-1} (t) \\ &= F_\l^{-1}(t) F_\l(t + t') B_\l(t+t') B_\l^{-1}(t)\,, \\
       \exp ( B_\l(t') \xi_\l B_\l^{-1}(t') \,\Sigma \lambda^{-i}t_i ) &= B_\l(t') F_\l(t) B_\l(t) B_\l^{-1} (t') \\ &=
       F_\l^{-1}(t') F_\l(t + t') B_\l(t+t') B_\l^{-1}(t')\,.
\end{split}
\end{equation*}
Set $\hat{B}_\l(t) = B_\l(t+t') B_\l^{-1}(t')$ and $\tilde{B}_\l(t') =  B_\l(t+t') B_\l^{-1}(t)$. Then
\begin{equation*} \begin{split}
	\pi(t+t') \xi _\l&= B_\l(t+t') \,\xi_\l \, B_\l^{-1}(t+t') \\
	&= \tilde{B}_\l(t')B_\l(t) \,\xi_\l \,B_\l^{-1}(t) \tilde{B}_\l^{-1}(t') \\
	&=\tilde{B}_\l(t')\, \pi(t)\xi_\l \,\tilde{B}_\l^{-1}(t') = \pi(t')\pi(t) \xi _\l\\
	&= \hat{B}_\l(t)B_\l(t')\,\xi_\l \,B_\l^{-1}(t') \hat{B}_\l^{-1}(t) \\
	&=\hat{B}_\l(t)\, \pi(t') \xi_\l \,\hat{B}_\l^{-1}(t) = \pi(t)\pi(t') \xi_\l\,.
\end{split}
\end{equation*}
\end{proof}
\subsection*{Removing singularities on $\Sigma$.} Elements of different isospectral sets may give the same extended frame up to conjugation by $g \in \bbT$ (see remark \ref{real} and below). This is the case in particular if an initial value $\xi_\l$ has a root at some $\lambda=\alpha_0 \in\C ^\times$. Then the corresponding polynomial Killing field $\zeta_\l$ also has a root at $\l=\alpha_0$ for all $z\in\C$. In this case we may reduce the order of $\xi_\l$ and $\zeta_\l$ without changing the immersion. This situation corresponds to a singular spectral curve because then the polynomial $a (\l) = -\lambda \det \xi_\l$ has a root of order at least two at $\alpha_0$. We can remove such a singularity by changing the surface by an isometry. We describe this change below.
\begin{proposition} \label{th:pkf-reduce}
Suppose a polynomial Killing field $\zeta_\l$ has roots in $\lambda \in\C^\times$. Then there is a polynomial $p(\lambda)$ such that $\zeta_\l/p (\l)$ has no roots in $\lambda \in \C^\times$. If $F_\l$ and $\breve F_\l$ are the extended frames of $\zeta_\l$ respectively $\zeta_\l/p (\l)$, then $\breve F _\l (p(0)\,z)= F_\lambda (z)$. Hence there is $g(\delta) \in \bbT$ with $\tilde \zeta_\l: = g(\delta)^{-1}(\zeta_\l/p (\l) )\,g(\delta):\C \to \pk$ and $\tilde F (z):= g(\delta)^{-1} \breve {F} _\l (p(0)z) g(\delta)$ induces a minimal surface $\tilde X_\l$ congruent to $X_\l$ in $\SY^2 \times \R$.
\end{proposition}
\begin{proof}
Suppose the polynomial Killing field $\zeta_\l (z) = F_\l^{-1}(z) \xi_\l F_\l (z)$ has a root at $\l = \alpha_0$. Define
\begin{align*}
  p(\lambda)&=\begin{cases}
 \overline{  \sqrt{-\alpha_0}} \lambda + \sqrt{-\alpha_0}
  &\mbox{ if } \alpha_0 \in \bbS^1\\
  (\lambda-\alpha_0)(1- \bar \alpha_0\lambda)&
  \mbox{ if }\alpha_0 \in\C^\times \backslash \SY^1\,.\end{cases}
\end{align*}
If $\zeta_\l$ has a simple root at $\lambda=\alpha_0 \in\C^\times$, then $\zeta_\l /p(\l) :\C \to  \pkg{g- \deg p}$ does not vanish at $\alpha_0$.
Then there is a map $C:\C \to \Lambda^+{\rm SL}_2 (\C)$ such that
\begin{align*}
    F_\l (z)\,B_\l (z)  = \exp (z\,\xi_\l ) & =  \exp (p(\l)\,z \,\breve{\xi}_\l) =  \exp( p(0)\,z\, \breve{\xi}_\l)\,C(z) \\ &= \breve{F} _\l (p(0)z) \breve B_\l (p(0)z)
\end{align*}
and
\begin{align*}
  F _\l (z) = \breve{F}_\l (p(0)\,z)  & =\begin{cases}
    \breve{F} _\l ( \sqrt{-\alpha_0}z)
  &\mbox{ if } \alpha_0 \in \bbS^1\\
 \breve{F} _\l (- \alpha_0 z) &
 \mbox{ if }\alpha_0 \in\C^\times \backslash \SY^1\,. \end{cases}&
\end{align*}
We consider $\delta =\overline{  \sqrt{-\alpha_0}}$ when $|\a_0|=1$ and $\delta=- \bar \alpha_0 / | \alpha_0 |$ in the other case.
We conjugate $\zeta_\l /p(\l)$ by $g(\delta)\in \bbT$. Hence by remarks \ref{IN1}, \ref{IN2}, \ref{IN3}, \ref{IN4}, the immersion $\tilde X_\l$ obtained from $\tilde F_\l (z)= g(\delta)^{-1} \breve{F}_\l (p(0)\,z) \,g(\delta)$ is congruent to the immersion $X_\l$.
\end{proof}
Hence amongst all polynomial Killing fields that give rise to an extended frame of finite type there is one of smallest possible degree.
We say that  a polynomial Killing field has \emph{minimal degree} if and only if it has neither roots nor poles in $\lambda \in \C^\times$.  We summarize two results by Burstall-Pedit \cite{BurP_adl,BurP:dre}:
\begin{proposition}
For an extended frame of finite type there exists a unique polynomial Killing field of
minimal degree. There is a smooth 1-1 correspondence between the set of extended frames of finite type and the set of polynomial Killing fields without zeroes.
\end{proposition}
\begin{proof}
We briefly outline how to prove the existence and uniqueness of a
minimal element. If the initial value $\xi_\l$ gives rise to an extended frame $F_\l$,
then the corresponding polynomial Killing field $\zeta_\l$ can be modified according to
Proposition \ref{th:pkf-reduce} so that $\tilde{\zeta}_\l$ is of minimal
degree, and still giving rise to $F_\l$. Hence there exists a
polynomial Killing field $\zeta_\l \in \pk$ of least degree giving rise to $F_\l$.

For the uniqueness, let $\zeta_\l$ and $\tilde \zeta_\l$ both solve $d \zeta_\l + [\a(\o), \zeta_\l]=0$, with ${\rm deg}\zeta_\l \geq {\rm deg} \tilde \zeta_\l$. We can assume that $\zeta_\l, \tilde \zeta_\l$ have no roots (if not, we simplify the polynomial Killing field using Proposition~\ref{th:pkf-reduce}). We use the iteration of Proposition~\ref{th:PinS-iteration}. We prove that there is a polynomial $q \in \C^k[\l]$ such that
\begin{equation}
\label{unique}
    \zeta_\l (z)= q(\l)\, \tilde \zeta_\l(z)\,.
\end{equation}
The polynomial $q$ is constructed recursively by considering
coefficients $u_n$, $\sigma_n$, $\tau_{n-1}$ and $\tilde u_n$, $\tilde \sigma_n$, $\tilde \tau_{n-1}$. Since $\tau_{-1}, \tilde \tau_{-1}$, are constant, there is $q_0$ with $ \tilde \tau_{-1}=q_0 \tau_{-1}$. This implies that  $ \tilde u_{0}=q_0 u_{0}$, $ \tilde \sigma_{0}=q_0 \sigma_{0}$ and there is
$q_1$ such that $ \tilde \tau_{0}=q_0 \tau_{0}+q_1 \tau_{-1}$. By the iteration, if there are constants $q_0,q_1,...,q_{\ell}$ with
\begin{align*}
&\tilde u_{0}=q_0 u_{0}\,, &\tilde u_{1}=q_0 u_{1}+q_1 u_{0}\,, \quad &\ldots \quad \tilde u_{\ell}=q_0 u_{\ell}+...+q_{\ell} u_{0}\,, \\
&\tilde \sigma_{0}=q_0 \sigma_{0}\,, &\tilde \sigma_{1}=q_0 \sigma_{1} + q_1 \sigma_{0}\,, \quad &\ldots \quad \tilde \sigma_{\ell}=q_0 \sigma_{\ell} +...+ q_{\ell} \sigma_{0}\,, \\
&\tilde \tau_{-1}=q_0 \tau_{-1}\,, &\tilde \tau_{0}=q_0 \tau_{0} + q_{1} \tau_{-1}\,, \quad &\ldots \quad \tilde \tau_{\ell-1}=q_0 \tau_{\ell-1} +...+ q_{\ell} \tau_{-1}\,.
\end{align*}
the iteration implies that there is $q_{\ell+1}$ such that $\tilde \tau_{\ell}= q_0 \tau_{\ell}+...+q_{\ell+1} \tau_{-1}$ and this proves (\ref{unique}).
Now, since $\zeta_{\l}, \tilde \zeta_{\l}$ have no roots, the polynomial $q(\l)=q_0$ is constant and since the residues coincide, we conclude that $\zeta_{\l}=\tilde \zeta_{\l}$.
\end{proof}
\begin{remark}
Since the Iwasawa factorization is a diffeomorphism, and all other operations involved in obtaining an extended frame from a polynomial Killing field are smooth, the resulting minimal surface depends smoothly on the entries of the polynomial Killing field, and thus also smoothly on the entries of its initial value.
\end{remark}
For $\xi_\l \in I(a)$ the polynomial $a(\l) = -\lambda \det \xi_\l$ has the form
\begin{equation} \label{eq:a-expansion}
    a(\l) = b\,\prod^{g}_{i=1} (\l - \a_i)(1-\l \bar \a_i)\,,\quad b \in \R^-\,.
\end{equation}
The condition \eqref{eq:hermit} implies $a(1) \leq 0$ so that $b \in \R^-$.
\begin{lemma}
{\rm(1)} If $a$ has a double root $\alpha_0 \in \bbS^1$, then we have an isomorphism
\[
  I(a) \cong I(- \alpha_0 \, ( \l-\alpha_0)^{-2} \,a) \,.
\]
{\rm(2)} If $a$ has double root $\alpha_0 \notin \bbS^1$ then $I(a)=\{\xi _\l \in I(a) \mid \xi _{\alpha_0} \neq 0\} \cup \{\xi_\l  \in I(a) \mid \xi _{\alpha_0}=0\}$
and there is an isomorphism
\[
\{ \xi _\l \in I(a) \mid \xi _{\alpha_0} =0\} \cong  I((\l-\alpha_0)^{-1}(1-\bar{\alpha}_0 \l)^{-1} \,a)\,.
\]
\end{lemma}
\begin{proof} {\rm (1)} If $a$ has a double root at $\alpha_0 \in \bbS^1$ then for any $\xi_\l \in I(a)$, we have $\xi _{\alpha_0}=0$. We can remove this root by Proposition \ref{th:pkf-reduce}, with $\delta=\overline{\sqrt{-\alpha_0}}$, to obtain the isomorphism
$$
  \xi _\l \mapsto (\overline{\sqrt{-\alpha_0}} \lambda +\sqrt{ -\alpha_0})^{-1} g(\delta)^{-1}  \xi _\l\; g(\delta) \in I(-\alpha_0 \, ( \l-\alpha_0)^{-2} \,a)\,.
$$
{\rm (2)} Suppose $a$ has a double root at $\alpha_0 \notin \bbS^1$. Then the isospectral set splits into a part which contains potentials with a zero at $\alpha_0$, which we can remove, and the set of potentials not zero at $\alpha_0$. But in this last case, this means that $\xi _{\alpha_0}$ is a nilpotent matrix. We use  Proposition \ref{th:pkf-reduce}, with $\delta=- \bar \alpha_0 / |�\alpha_0 |$, to obtain the isomorphism

$$ \xi_\l \mapsto \tfrac{1}{\l-\alpha_0}\,\tfrac{1}{1-\bar \alpha_0 \l} \,g(\delta)^{-1}  \xi _\l\; g(\delta) \in   I(\tfrac{1}{\l-\alpha_0}\,\tfrac{1}{1-\bar{\alpha}_0 \l}\,a)\,.
$$
\end{proof}
\begin{theorem} \label{th:isospectral-set-properties}
Isospectral sets are compact. If the $2g$-roots of the polynomial $a(\lambda) = -\lambda\,\det \xi_\l$ for $\xi_\l \in \mathcal{P}_g$ are pairwise distinct, then $I(a) \cong (\bbS^1)^g$ is a connected smooth g-dimensional manifold diffeomorphic to a $g$-dimensional real torus.
\end{theorem}
The proof of Theorem \ref{th:isospectral-set-properties} follows in several steps. The compactness is a consequence of the next Proposition \ref{proper}. The second statement follows from Propositions \ref{maxrank} and \ref{immersion}. We shall prove several properties of the map
\begin{equation} \label{eq:map-A}
    A : {\mathcal{P}}_g \to {\mathcal{M}}_g,\,\xi_\l \mapsto -\l \det \xi_\l\,.
\end{equation}
\begin{proposition}
\label{proper}
The map $A$ in \eqref{eq:map-A} is proper.
\end{proposition}
\begin{proof}
Since $A$ is continuous it suffices to show that pre-images of compact sets are bounded. Let $K \subset {\mathcal{M}}_g$ be compact. Then the map $\bbS^1 \times K \to \R,\, (\l,\, a) \mapsto \l^{-g} a(\l)$ is bounded. For the compactness of the pre-image it suffices to show that all Laurent-coefficients of a $\xi_\l \in {\mathcal{P}}_g$ are bounded, if $A(\xi_\l) \in K$. Fix a polynomial $a \in K$ and consider the isospectral set $I(a)$ as a closed subset of the $(3g+1)$-dimensional vector space ${\mathcal{P}}_g$. For $d=(1-g)/2$ the map $\lambda^d\xi _\l$ is traceless and skew-hermitian for $|\lambda|=1$. The determinant of traceless skew-hermitian $2 \times 2$ matrices is the square of a norm $\| \cdot \|$. The Laurent-coefficients of $\xi _\l= \sum_{i=-1}^g \l^i \ \hat \xi_i $ are
$$
 \hat   \xi_i = \frac{1}{2\pi\mi}\int_{\SY^1} \l ^{-i} \xi _\l\,\frac{d\l}{\l}\,.
$$
Using the norm gives
$$
    \| \hat \xi_i \| \leq \frac{1}{2\pi\mi}\int_{\SY^1} \| \l ^{(1-g)/2} \xi_\l \| \,\frac{d\l}{\l}  \leq \sup_{\l \in \SY^1} \sqrt{-\l^{-g}a(\l)}\,.
$$
Thus each entry of  $\hat \xi_i$ is bounded on $\bbS^1$, so $A$ is proper and $I(a)$ therefore compact.
\end{proof}
\begin{proposition} \label{maxrank}
Suppose $\xi_\l \in {\mathcal{P}}_g$ has no roots in $\lambda \in \C^\times$. Then the map $A$ in \eqref{eq:map-A} has maximal rank $2g+1$. Let $a(\l) =-\l \det \xi_\l$. Then $I(a)$ is a g-dimensional sub-manifold of ${\mathcal{P}}_g$.
\end{proposition}
\begin{proof}
Since $\det$ is the square of a norm on $\su$, at all roots of $a$ on $\bbS^1$, the corresponding $\xi_\l \in I(a)$ has to vanish. If $\xi_\l$ is without roots, then $a \in {\mathcal{M}}_g^0$. We show that the derivative of the map $A$ has rank 2g+1 at a potential $\xi_\l$ without roots, and then invoke the implicit function theorem. Hence it suffices to prove that for all roots $\alpha_0$ of $a$ of order $n$, and all $\beta \in \C$ there exists a tangent vector $\dot \xi_\l$ along ${\mathcal{P}}_g$ at $\xi_\l$, such that the corresponding derivative of $a$ is equal to
$$
    \dot a (\l) = - \l \det \xi_\l \, { \rm tr}(\xi_\l^{-1} \dot \xi_\l) = \frac{\beta a(\l)}{(\l-\alpha_0)^m} +\frac{\l ^m \bar \beta a(\l)}{(1-\bar \alpha_0 \l)^m}
$$
with $m=1,\,\ldots,\,n$. Such a vector field fixes all roots of $a$ except $\alpha_0$. The set of these vector fields form a 2g-dimensional real vector space.
Besides the roots of $a(\l)$ also the coefficient $b$ in (4.3) can be  changed by this variation. We consider $\dot \xi _\l= t\,\xi_{\l}$ with
$$
    \dot a (\l) = - \l \det \xi_\l {\rm tr}(\xi_\l^{-1} \dot \xi_\l)=-2 \l  t \det \xi _\l\,.
$$
This vector field preserves all roots of $a(\l)$, but changes $a(1)$ with variational field $\dot a (1) = t a(1)$, and $b \in \R^-$ is changing non-trivially. This will prove the theorem.

Now we construct vector fields $\dot \xi_\l$.  If $\alpha_0$ is a root of $a$ of order $n$, then $\det\xi_{\alpha_0}$ vanishes, and $\xi_{\alpha_0}$ is nilpotent. For a nonzero nilpotent $2\times 2$-matrix $\xi_{\alpha_0}$ there exists a matrix $Q \in \su$ such that $\xi_{\alpha_0} =[Q,\,\xi_{\alpha_0}]$.
To prove this remark, observe that it holds if $\xi_{0} = \bigl( \begin{smallmatrix} 0 & 1 \\ 0 & 0 \end{smallmatrix} \bigr)$, by setting $Q_0 = \bigl( \begin{smallmatrix} 1 & 0 \\ 0 & -1 \end{smallmatrix} \bigr)$. The general statement now follows since there exists $g \in \SU$ with $\xi _{\alpha_0} = g^{-1} \xi_{0} g$, and setting $Q=g^{-1} Q_0 g$.

We need the following basic fact: For $A,\,B \in \Sl (\C)$ and $A \neq 0$, we have
\begin{equation} \label{eq:crochet}
    {\rm tr}(AB)=0 \Longleftrightarrow B=[C,A] \,\,\hbox{ with some } C \in \Sl (\C)\,.
\end{equation}
For any $\xi_\l \in {\mathcal{P}}_g$, we have $\xi _\l^2 = -\det (\xi_\l) \,\un$ and at a root $\alpha_0$ of $a(\l)$ of order $n$, there exists for any $m=1,...,n$, a matrix
$$
\hat Q_\l (m)=Q_0 + Q_1 (\l - \a_0)+...+Q_{m-1} (\l - \a_0)^{m-1}
$$
with $Q_0,\,\ldots,\,Q_{n-1} \in \Sl (\C)$, such that $\xi_\l -[\hat Q_\l (m),\,\xi_\l]$ has at $\alpha_0$ a root of order $m$. The matrix $\hat Q_\l (m)$ is constructed inductively using \eqref{eq:crochet}.
We remark that at $\a_0$, the function  $\l \mapsto {\rm tr} (\xi_{\l}^2)$ has a root of order $n$.
Then there is $Q_0$ such that $\xi _\l - [Q_0,\,\xi_\l]= (\l - \a_0) \xi_{\l,1}$  and ${\rm tr} ( \xi_\l(  \xi_\l - [Q_0,\,\xi_\l]))=(\l - \a_0){\rm tr} ( \xi_\l  \xi_{\l,1}) $ has a root at $\a_0$ of order $n$. Then there is a matrix $Q_1$ with  $\xi_{\l,1}-[Q_1,\,\xi_\l]=(\l - \a_0)Q_2$. Now we define for $m=1,...,n$
$$
\tilde Q_\l (m) = - {}^t \bar Q_0 - {}^t \bar   Q_1 (1-\bar \a_0 \l)\l^{-1} -...-{}^t \bar  Q_{m-1} (1- \bar \a_0 \l)^{m-1}\l^{m-1}.
$$
Then $\xi_\l -  [\tilde Q _\l (m) , \,\xi_\l]$ has at $\bar \a_0 ^{-1}$ a root of order $m$. Now we define
\begin{equation*}
\begin{split}
    q_m(\l) &= \tfrac{\beta}{(\l-\a_0)^m} + \frac{\l^m \bar \beta}{(1-\a_0\l )^m}\,, \\
    Q_\l (m) &= \tfrac{\beta}{(\l-\a_0)^m}\,\hat Q_\l (m) + \tfrac{\l^m \bar \beta}{(1-\a_0\l )^m}\, \tilde Q_\l (m)\,.
\end{split}
\end{equation*}
There exists some $P \in \su$ such that
$$
    \dot \xi_\l =q_m(\l) \xi_\l - [Q_\l (m),\,\xi_\l] +[P,\,\xi_\l] \in T_{\xi_\l}{\mathcal{P}}_g \,.
$$
To see that we need to check
$\dot {\hat \xi}_{-1} \in \mi\mathbb{\R^+} \bigl( \begin{smallmatrix} 0 & 1 \\ 0 & 0 \end{smallmatrix} \bigr)$ and $\dot {|a|}(0) = 0$. Note that if $A \in \su$ and $B=\bigl( \begin{smallmatrix} 0 & 1 \\ 0 & 0 \end{smallmatrix} \bigr)$, then $[A,\,B ]= \bigl( \begin{smallmatrix} \alpha & \mi x \\ 0 & -\alpha \end{smallmatrix}\bigr)$ with $\a \in \C,\, x \in \R$. Then we can choose $P \in \su$ such that
$$
\dot {\hat \xi}_{-1}=\frac{\beta \xi_{-1} -\beta [Q_0 (m), \,\xi_{-1}]}{(-\a_0)^n} + [P,\, \xi_{-1}] \in \mi \R^+\begin{pmatrix} 0 & 1 \\ 0 & 0 \end{pmatrix}\,.
$$
For the second condition we have $(\dot a \bar a + a \bar{\dot a} )(0)= |a|^2(0) (-\a_0)^{-n}(\beta + \bar \beta)$, and we can choose $\beta \in \C$ such that the variational field keeps $|a(0)|$ unchanged along the deformation. This proves that $\dot \xi_\l \in T_{\xi_\l}{\mathcal{P}}_g$, and such vector fields span a $2g+1$ dimensional real vector space of vector fields in the complement of the kernel of the map. This proves that around $\xi_\l$ without zeroes, $I(a)$ is a real $g$-dimensional manifold.
\end{proof}
\begin{proposition}
\label{immersion}
For all $\xi _\l\in {\mathcal{P}}_g$ without roots the vector fields of the isospectral group action generate at $\xi_\l$ a real $g$-dimensional subspace of the tangent space ${\mathcal{P}}_g$ at $\xi_\l$.
\end{proposition}
\begin{proof}
The vector field $(t_0,\,\ldots,\,t_{g-1})$ of the isospectral action at $\xi_\l$ takes the values
$$
    \dot \xi_\l = \bigl[ \bigl( \sum_{i=0}^{g-1} \l^{-i}t_i \,\xi_\l  \bigr)^+, \,\xi_\l \bigr] = - \bigr[ \bigl( \sum_{i=0}^{g-1} \l^{-i}t_i \,\xi_\l \bigr)^-, \,\xi_\l \bigr]\,.
$$
Here
$$
\sum_{i=0}^{g-1} \l^{-i}t_i \,\xi_\l  = \bigl( \sum_{i=0}^{g-1} \l^{-i} t_i \,\xi_\l \bigr)^+ + \bigl( \sum_{i=0}^{g-1} \l^{-i} t_i \,\xi_\l  \bigr)^-
$$
is the Lie algebra decomposition of the Iwasawa decomposition. For $A \in \Sl (\C)$ with $A \neq 0$ we have that $\{ B \in \Sl (\C) \mid [A,\,B]=0\} = \{ xA \mid x \in \C\}$. Hence the vector field corresponding to $(t_0,\,\ldots,\,t_{g-1})$ vanishes at $\xi_\l$ if and only if there exists a decomposition of the polynomial $\sum_{i=0}^{g-1} \l^{-i}t_i =f^+ (\l) + f^- (\l)$ into complex functions such that
$$
\bigl( \sum_{i=0}^{g-1} \l^{-i}t_i \, \xi_\l \bigr)^+ = f^+ \xi_\l  \quad \hbox{ and } \quad \bigl( \sum_{i=0}^{g-1} \l^{-i}t_i \,\xi_\l \bigr)^- = f^- \xi_\l \,.
$$
Hence $\overline{ f^+ (\bar \l ^{-1})} = \l^g f^+ (\l)$ and $f^- (0)=0$.

The polynomial $\sum_{i=0}^{g-1} \l^{-i}t_i$ is a linear combination of such functions if and only if $t_{g-1-i}=\bar t_i$. The subspace of such $(t_0,\,\ldots,\,t_{g-1})$  is a real $g$-dimensional subspace of $\C^g$. This implies the proposition.
\end{proof}
Recall ${\mathcal{M}}_g$ from \eqref{eq:calM-sets}, and define
\begin{align*}
    {\mathcal{M}}_g^1 &= \{ \,a \in {\mathcal{M}}_g \mid \mbox{ $a$ has $2g$-pairwise distinct roots }\,\}\,, \\
    {\mathcal{P}}_g^{1} &=\{\, \xi_\l \in {\mathcal{P}}_g \mid a(\l)=-\l \det \xi_\l \in {\mathcal{M}}_g^{1} \}\,.
\end{align*}
\begin{proposition}
For all $a \in {\mathcal{M}}_g^1$, the isospectral action $\pi : \R^g \times I(a) \to I(a)$ acts transitively on $I(a)$ and the mapping $A: {\mathcal{P}}_g^{1} \to {\mathcal{M}}_g^{1}$ is a principal bundle with fibre $I(a) = (\bbS^1)^g$.
\end{proposition}
\begin{proof}
At all roots of $\xi_\l \in {\mathcal{P}}_g$, the determinant $\det \xi_\l$ has a higher order root. Hence all $\xi_\l \in {\mathcal{P}}_g^{1}$ have no roots on $\C^\times$. Proposition \ref{maxrank} implies that $A: {\mathcal{P}}_g^1 \to {\mathcal{M}}_g^1$ has maximal rank $2g+1$ and induces a fibre bundle, whose fibres are real $g$-dimensional manifolds. The isospectral action preserves the determinant and thus the fibres. Proposition \ref{immersion} implies that for all $\xi_\l \in {\mathcal{P}}_g^1$ the corresponding orbit of the isospectral action is an open submanifold of the corresponding fibre. If $\pi(t_n)\,\xi_\l$ converges to $\tilde\xi_\l \in {\mathcal{P}}_g ^{1}$ for a sequence $(t_n)_{n \in \N}$ in $\R^g$, then the orbit of $\tilde\xi_\l$ is again an open submanifold of the corresponding fibre. Thus $\pi(t_n)\,\xi_\l$ belongs to the orbit of $\tilde\xi_\l$ for sufficiently large $n \in \N$. Then $\pi(t_n)\,\xi_\l = \pi(t'_n)\,\tilde\xi_\l$ and $\tilde\xi _\l= \pi(t_n-t'_n)\,\xi_\l$ is in the orbit of $\xi_\l$. This shows that the orbits of the isospectral set $I(a)$ are open and closed submanifolds of the fibre. Due to Proposition \ref{proper} the fibers are compact. Therefore all orbits are compact as well. Due to Proposition \ref{immersion}, for all $\xi_\l \in {\mathcal{P}}_g^1$ the stabilizer subgroup
\begin{equation} \label{eq:stabilizer}
    \Gamma_{\xi_\l} = \{ t \in \R^g \mid \pi (t)\,\xi_\l = \xi _\l\}
\end{equation}
is discrete and $\R^g / \Gamma_{\xi_\l}$ is diffeomorphic to the connected component of the fibre of $\xi_\l$. We conclude that $\Gamma_{\xi_\l}$ is a lattice in $\R^g$ isomorphic to $\Z^g$. This shows that $A: {\mathcal{P}}_g^1 \to {\mathcal{M}}_g^1$ is a fibre bundle, whose fibres have connected components all isomorphic to $(\bbS^1)^g$.

It remains to prove that $I(a)$ has only one connected component. We first show that ${\mathcal{M}}_g^0$ (see \eqref{eq:calM-sets}) is path connected: If $a,\,\tilde a \in {\mathcal{M}}_g^0$ satisfy
\begin{equation} \label{eq:path-connected}
    t\,\tilde a(0) + (1-t)\,a(0) \neq 0 \quad \mbox{ for all } t\in[0,\,1]
\end{equation}
then $t\,\tilde a + (1-t)\,a \in {\mathcal{M}}_g^0$ for all $t\in[0,\,1]$.
If $a,\,\tilde a \in {\mathcal{M}}_g^0$ do not satisfy \eqref{eq:path-connected}, then modifying $\lambda \mapsto a(\lambda)$ by the rotation $\lambda \mapsto e^{-\mi g \theta} a(e^{\mi\theta} \lambda)$ for some suitable $\theta \in \R$ we can ensure that \eqref{eq:path-connected} holds. Hence ${\mathcal{M}}_g^0$ is path connected.

Since ${\mathcal{M}}_g^1$ is an open subset of ${\mathcal{M}}_g^0$, whose complement ${\mathcal{M}}_g^0\setminus{\mathcal{M}}_g^1$ has codimension at least $2$, we conclude that ${\mathcal{M}}_g^1$ is connected and $A: {\mathcal{P}}_g^1 \to {\mathcal{M}}_g^1$ has maximal rank $2g+1$ at any point. Hence it remains to show that there exists at least one $a \in {\mathcal{M}}_g^1$ for which $I(a)$ has only one connected component.

Denote the entries of $\xi_\l \in {\mathcal{P}}_g$ by polynomials $\a, \beta, \gamma$ so that $$
    \xi_\l =  \begin{pmatrix} \a(\l) & \beta(\l) \\ \gamma(\l) & -\a(\l) \end{pmatrix}\,.
$$
Then $a(\l) = -\lambda \det \xi_\l = \l \a^2 + \l\beta \gamma$. Here $\a$ is a polynomial of degree at most $g-1$. For $|\l|=1$ the polynomial $\l^{\frac{1-g}{2}} \a \in \mi\R$ and $\l^{1-g} \a^2 \leq 0$, and therefore $\l^{1-g}\beta \gamma \in \R$.

Now we consider a potential $\xi_\l$ with $\l \beta= \gamma$, and $\gamma$ has only roots on $|\l|=1$. We claim in this case that $I(a)$ is connected. This condition on $\beta$ implies that $a(\l)=\l \alpha^2 +\gamma^2$ and $\l^g  \overline{\gamma ( \bar \lambda ^{-1})}=-\l \beta(\l)=-\gamma (\l)$. Now observe that $\l^{-g}\gamma^2 \leq 0$ and $\l^{-g}  a(\l) \leq \l^{1-g} \a^2 \leq 0$.

Let $\tilde \xi_\l \in I(a)$ with entries $\tilde \a, \tilde \beta, \tilde \gamma$. We construct a family $\xi_{\l,t} = \bigl( \begin{smallmatrix} \a_t & \beta_t \\ \gamma_t & -\a_t \end{smallmatrix} \bigr)$  with $\a_t= t \a + (1-t) \tilde \a$. We prove that there exist polynomials $\beta_t, \gamma_t$ uniquely defined such that $\xi_{\l,t} \in I(a)$ for all $t \in [0,\,1]$,
with $\beta_1=\beta$, $\gamma_1=\gamma$ and $\beta_0 = \tilde \beta$, $\gamma_0=\tilde \gamma$.

Since $\l^{1-g} \tilde \a ^2 \leq 0$ for $|\l|=1$ we have  $\l^{1-g}  \a_t ^2 \leq 0$ for all $t \in [0,\,1]$. Now we consider the polynomial $p_t(\l)=a(\l)-\l^{1-g}  \a_t ^2$. If $\a_0$ is a root of $p_t(\l)$, then $\bar \a_0^{-1}$ is a root of $p_t(\l)$.
For $t=0$ we know that $p_0=\l \tilde \beta\, \tilde \gamma$ where roots of $\tilde \beta$ are symmetric to the roots of $\tilde \gamma$. At roots of $p_t$ we can define $\beta_t$ and $\gamma_t$ with $\beta_1 \in \mi\R^+$.
At $t=1$, we have $p=\gamma^2$ and all the roots of $p_1$ are double roots on the unit circle. Then $\gamma$ is defined uniquely, and $I(a)$ is connected, if we can find such a path with $a \in {\mathcal{M}}_g^1$.

Therefore we consider $\gamma= \l^{g}+1$ and $\a=\mi k(\l^{g-1}+1)$ with $k \in \R^+$. Then $0=\l^{g}+1=\l^{g-1}+1$ implies $\l^{g-1}(1-\l)=0$. Then the polynomials $\gamma$ and $\a$ do not have common roots. Hence at $k=0$ we have $\l^{-g} a(\l)=\l^{1-g} \a^2 + \l^{-g} \gamma^2 \leq 0$ for $|\l|=1$, and $a$ has only double roots on the unit circle.
For $k>0$ small enough, the roots change. But there are no roots on the unit circle and the roots are simple and conjugate, so that $a \in {\mathcal{M}}^1_g$.
\end{proof}


\section{Periods}

Suppose $X_\lambda = ( G_\lambda,\, h_\l)$ is an associated family of minimal surfaces in $\bbS^2 \times \bbR$. For one member of this family to be periodic, say for $\lambda = 1$ with period $\tau \in \C^\times$, this means that $X_1 (z+\tau) = X_1(z)$ for all $z \in \C$, or equivalently $G_1(z+\tau)=G_1(z)$ and $h_1(z + \tau)= h_1(z)$ for all $z \in \C$. If $G_\lambda (z) = F_\lambda (z)\, \sigma_3\, F_\lambda (z)^{-1}$ and $h_\lambda (z) = {\rm Re}(-4 \mi \sqrt{\beta_{-1} \gamma_0 \lambda^{-1}}\,z)={\rm Re}(- \mi e^{\mi\Theta/2}\sqrt{ \lambda^{-1}}\,z)$ (where $Q=\tfrac{1}{4}e^{\mi\Theta} (dz)^2$) then periodicity reads
\[
    \bigl[ F^{-1}_1 (z)F_1 (z+\tau),\,\sigma_3 \bigr] = 0 \quad \mbox {and} \quad
        {\rm Re}(-\mi  e^{\mi\Theta/2}\,\tau) = 0\,.
\]
The {\bf monodromy} of an extended frame $F_\l$ with respect to the period $\tau$ is the matrix
$$
    M_\l (\tau) = F_\lambda (z)^{-1}\,F_\lambda (z+\tau)\,.
$$
Thus periodicity of the horizontal part reads $[M_\l (\tau),\,\sigma_3] =0$.
Due to \eqref{eq:symb}, the monodromy $\C^\times \to \SL
(\C),\,\lambda \mapsto M_\l (\tau)$ is a holomorphic map with
essential singularities at $\lambda = 0,\,\infty$.

For a periodic immersion its conformal factor $\o$ is periodic, and hence also $\a(\o)$ in \eqref{eq:symb} is periodic.
This in turn implies that $dM_\l (\tau) = 0$ so that $M_\l (\tau)$ does not depend on $z$. Hence $M_\l (\tau) = F_\lambda (0)^{-1}\,F_\lambda (\tau) = F_\lambda (\tau)$ since $F_\l (0)= \un$.

Let $\zeta_\l$ be a periodic solution of the Lax equation \eqref{eq:lax-eq} with initial value $\xi_\l\in\pk$, with period $\tau$ so that $\zeta_\l(z+\tau) = \zeta_\l(z)$ for all $z \in \C$. Then also the corresponding $\alpha(\zeta_\l)$ in \eqref{eq:alpha-zeta} is $\tau$-periodic. Let $dF_\l = F_\l \alpha(\zeta_\l),\,F_\l(0)=\un$ and $M_\l (\tau) = F_\lambda(\tau)$ be the monodromy with respect to $\tau$. Then for $z=0$ we have $\xi_\l= \zeta_\l(0) = \zeta_\l(\tau) =
F^{-1}_\l (\tau) \,\xi_\l\,F_\lambda (\tau) = M^{-1} _\l (\tau)\, \xi_\l \,M_\l (\tau)$ and thus
$$
    [\,M_\l (\tau) ,\,\xi_\l\,] = 0\,.
$$
The monodromy takes values in $\SU$ for $|\lambda |=1$. The monodromy depends on the choice of base point, but its conjugacy class and hence eigenvalues $\mu,\,\mu^{-1}$ do not.
The eigenspaces of $M(\l,\,\tau)$ depend holomorphically on $(\nu,\, \l)$. The eigenvalues of $\xi_\l$ and $M_\l (\tau)$ are different functions on the spectral curve $\Sigma$.
\begin{proposition} \label{th:eigenvalues}
Let $M_\l = \bigl( \begin{smallmatrix} A & C \\ B & D \end{smallmatrix} \bigr) \in\Lambda\SU$ and $\xi_\l = \bigl( \begin{smallmatrix} \a & \beta \\ \gamma & -\alpha \end{smallmatrix} \bigr) \in \pk$ with $[M_\l,\,\xi_\l]=0$. Assume $\nu \neq 0$ and $\mu^2 \neq 1$. Then $M_\l$ and $\xi_\l$ have same eigenvectors $\psi_+ =(1,\, (\nu -\a)/\beta),\, \psi_- =(1,\, -(\nu + \a)/\beta )$ with
\begin{alignat*}{2}
    \xi_\l \psi_+ &= \nu\,\psi_+ &\hbox{ and }  \quad M \psi_+ &= \mu\, \psi_+\,,\\
    \xi_\l \psi_- &= - \nu\, \psi_- &\hbox{ and } \quad M \psi_- &= \mu^{-1} \,\psi_-\,.
\end{alignat*}
The involution $\eta : (\l,\nu) \to (\bar \l^{-1}, \bar \l^{1-g} \bar \nu)$ acts on $\mu$ by $\eta ^* \mu = \bar \mu$.
\end{proposition}
\begin{proof}
From the reality condition $\l^{g-1}\overline{\xi_{1/\bar \l}}{\,}^t=- \xi_\l$ or equivalently $\xi_{1/\bar \l}=- \bar \l^{1-g} \overline{\xi_\l}{\,}^t$ and with $\tilde{\psi}_-=(1,\,-(\bar \nu +\bar \a)/\bar \gamma)$ we have
$$
\xi_{1/\bar \l}\tilde{\psi}_- =- \bar \l^{1-g} \,\begin{pmatrix}  \bar \a & \bar \gamma \\ \bar  \beta & - \bar \alpha \end{pmatrix}\,\tilde{\psi}_- = \bar \l^{1-g} \bar\nu\, \tilde{\psi}_-\,.
$$
Since $\overline{M_{1/\bar \l}}^t{\,}^{-1} = M_\l$, and $ M_\l \psi_+=\mu\,\psi_+$, and $\psi_+ =(1,\,(\nu -\a)/\beta)=(1,\,  \gamma/(\nu +\a))$, we obtain the system
\[
    A+ C \frac{\gamma}{\nu +\a}= \mu  \quad \mbox{ and } \quad
    B+D  \frac{\gamma}{\nu +\a}= \mu  \frac{\gamma}{\nu +\a}\,.
\]
This implies by direct computation that $M_{1/\bar \l}\,\tilde{\psi}_- =  \overline{ M_{\lambda}^{-1} }^t \tilde{\psi}_- = \bar \mu \,\tilde{\psi}_-$.
\end{proof}
At $\l=0$ and $\l=\infty$ a monodromy $M_\l (\tau)$ has essential singularities. Next we study the behavior of $\mu = \mu(\lambda,\,\tau)$ at these two points.
\begin{lemma}\label{fdf}
Let $\psi = \psi(\lambda,\,z)=(1,\, h(\lambda,\,z))$ be an eigenvector of $\zeta_\lambda (z)= F_\lambda (z)^{-1} \,\xi_\l\,F_\lambda (z)$. Then there exists a complex function $f =f(\lambda,\,z)$ such that
\begin{equation}\label{eq:definitionf}
    f (\lambda,\,z)\,\psi(\lambda,\,z) = F^{-1} _\l (z) \,\psi (\lambda,\,0)
\end{equation}
which satisfies
\begin{enumerate}
\item $ f^{-1} df = -\frac{\mi}{4} \l^{-1/2} \exp(\mi\Theta/2) \,dz + O (1)$ in a neighborhood of $\l=0$\,,
\item $ f^{-1} df = -\frac{\mi}{4} \l^{1/2}  \exp(-\mi\Theta/2)\,d \bar z + O (1)$ in a neighborhood of $\l = \infty$.
\end{enumerate}
\end{lemma}
\begin{proof}
Note that $h(\l,\,z)$ and $f(\l,\,z)$ are 2-valued in $\lambda$ because they depend on the choice of eigenvalue. By Proposition \ref{th:eigenvalues} an eigenvector of
$$
    \zeta _\l (z) = \begin{pmatrix} \a (\lambda,\,z) & \beta (\lambda,\,z) \cr \gamma (\lambda,\,z) & -\a (\lambda,\,z) \end{pmatrix}
$$ associate to the eigenvalue $\nu =\sqrt {a(\l) \l^{-1}}$ is given by $\psi(\lambda,\,z) =\left(1,\,(\nu - \a)/\beta\right)$. Now using $\zeta_\l (z) = F^{-1}_\l (z) \,\xi_\l\,F_\l (z)$ we see that $F^{-1}  _\l (z) \,\psi(\lambda,\,0)$ is an eigenvector of $\zeta _\l (z)$ and it is collinear to $\psi (\lambda,\,z)$. This defines the function $f(\lambda,\,z)$.

Differentiating \eqref{eq:definitionf} reads $df\, \psi + f \, d \psi  = - \a(\zeta_\l) F_\l^{-1} \left.\psi\right|_{z=0}$, and then
\begin{equation} \label{eq:fvect}
f^{-1} d f\, \psi  = - \a (\zeta_\l) \,\psi  -d \psi \,.
\end{equation}
In a neighborhood of $\l =0$, we have $\nu ^2 = - \det \zeta_\l = \frac{\l \a^2 +  \l \beta \gamma}{\l}=\frac{\beta_{-1} \gamma_0}{\l} +O(1)=\frac{-e^{\mi\Theta}}{16 \l } + O(1)$. Expanding at $\l =0$
$$
\a(\zeta_\l) = \bigl( \begin{smallmatrix} \a_0 & \beta_{-1} \l ^{-1} \\ \gamma_0 & - \a_0 \end{smallmatrix} \bigr) dz + O(1)\,,
$$
and considering the first entry of the vector equation (\ref{eq:fvect}) yields
\begin{equation*} \begin{split}
    f^{-1} df &= -\a_0 dz - \beta_{-1} \l ^{-1} \frac{\nu - \a (z)}{\beta (z)} dz +O(1)\\
    &= -\nu dz  +O(1)=\frac{-\mi e^{\mi\Theta/2}}{4\sqrt{\l}}dz+O(1)\,.
\end{split}
\end{equation*}
In a neighborhood of $\l = \infty$, we have
$$
\a (\zeta_\l)= \bigl( \begin{smallmatrix} - \bar \a_0 & - \bar \gamma_0 \\ - \bar \beta_{-1} \l  &  \bar \a_0 \end{smallmatrix} \bigr) d \bar z + O(1)
$$
and we obtain (ii) by considering the eigenvector $\psi = \left(\tfrac{\beta}{\nu -\a },\,1 \right)$ and looking at the second entry
gives
\begin{equation*} \begin{split}
    f^{-1} df &=  \bar \beta_{-1} \l \frac{\beta}{\nu -\a} d \bar z + O(1) = \frac{\bar \beta _{-1} \beta_{g-1}}{\sqrt{-a_{2g}}}\sqrt{\l} d \bar z + O(1)\\
    &= \frac{ -\mi e^{-\mi\Theta/2}\sqrt{\l}}{4} d \bar z +O(1)\,.
\end{split}
\end{equation*}
Further, $\hat \xi _d =-\bar{ \hat \xi} ^t_{g-1-d}$ implies $\beta_{g-1}=-\bar \gamma_0$ and $a_{2g} = -\beta_{g-1}\gamma_g =- \bar \gamma_0 \bar\beta_{-1}= e^{-\mi\Theta}/16$.
\end{proof}
Using these properties we compute the local behavior of $\mu(\lambda,\,\tau)$ near $\l=0$ and $\l=\infty$.
\begin{proposition}
Let $X: A \to \SY^2 \times \R$ be an immersed finite type minimal cylinder with spectral curve $\Sigma$. Then there exists a meromorphic differential $d\ln\mu$ on $\Sigma$ with second order poles without residues at $\lambda=0,\,\infty$ so that $d \ln \mu -\tfrac{1}{4}\mi\tau \exp(\mi\Theta/2)\,d\sqrt{\l}\,^{-1}$ extends holomorphically to $\l=0$, and $d\ln\mu -  \tfrac{1}{4} \mi\bar\tau\exp(-\mi\Theta/2)\,d\sqrt{\l}$ extends holomorphically to $\l = \infty$.

This differential is the logarithmic derivative of a function $\mu:\Sigma \to \C$ which transforms under the involutions \eqref{eq:involutions} as $\sigma^\ast\mu=\mu^{-1}$, $\varrho^\ast\mu=\bar{\mu}^{-1}$ and $\eta^\ast\mu=\bar{\mu}^{}$. Further $\mu =\pm 1$ at each branch point of $\Sigma$.
\end{proposition}
\begin{proof}
If $\nu \neq 0$, the eigenspace of $\zeta_\l$ is a complex 1-dimensional vector space and since $[M_\l,\, \zeta_\l]=0$, every eigenvector $\psi$ of $\zeta_\l$ with associated eigenvalue $\nu \neq 0$ is an eigenvector of $M_\l$ with eigenvalue $\mu$ which depends only on $(\nu,\, \l)$. If $\nu=0$ then $\mu=\pm 1$. Note that $\mu$ is a non-zero holomorphic function on $\Sigma^*$. At $\l=0$ and $\l = \infty$, the monodromy has essential singularities and we thus need the local analysis of Lemma \ref{fdf}. If $\tau$ is the period of the annulus, we have $\psi (\tau)=\psi (0)$
and $f (\tau) \psi (0)=F_\l ^{-1} \psi (0)$. Then $\mu = f^{-1} (\tau)$. This proves that at $\l =0$, $d \ln \mu -\frac{\mi}{4}\, \exp(\mi\Theta/2)\, \tau \,d\sqrt {\l}^{-1}$ is holomorphic, and at $\l =\infty$ the differential
$d \ln \mu - \tfrac{\mi}{4}\,\exp(-\mi\Theta/2)\,\bar \tau \,d \sqrt{\l}$ extends holomorphically. The differential $d \ln \mu$ has second order poles without residues at $\l=0$ and at $\l=\infty$.

If $\psi_+$ and $\psi _-$ are eigenvectors associated to eigenvalues $\pm \nu$ of $\zeta_\l$,
the corresponding eigenvalues $\mu_\pm$ of $M_\l$ satisfy $\mu_+\mu_-=1$. To see how the involution $\eta$ acts on $\mu$, we note that since $M_\l$ satisfies \eqref{eq:reality-F}, we have that $\bar \mu$ is the corresponding eigenvalue of $M_{1/\bar \l}$ associated to $\mu$ by Proposition \ref{th:eigenvalues}. Thus $\eta^\ast\mu=\bar{\mu}^{}$.
Similarly $\varrho^* \mu = \bar \mu^{-1}$.

If $\nu =0$, then $\det \zeta_\l=0$, so let $(\tilde e_1,\,\tilde e_2) \in \C^2$ such that $\zeta_\l (\tilde e_1)=0$ and $\zeta_\l (\tilde e_2)=\delta \tilde e_1$. Since $M_\l \in \SL (\C)$, let $(e_1,\,e_2)$ be a basis of eigenvectors of $M_\l$ associate to $\mu$ and $\mu^{-1}$ and $\zeta_\l (e_1) =\kappa_1 \delta \tilde e_1$, $\zeta_\l (e_2) = \kappa_2 \delta \tilde e_1$. Inserting this in $M_\l \zeta_\l (e_i)=\zeta_\l M_\l(e_1)$ proves that $\mu = \mu ^{-1}$. This proves that the holomorphic function $\mu$ takes values $\pm 1$ at each branch point of $\Sigma$.
\end{proof}
Next we relate the eigenvalues $\mu$ to the isospectral action. We prove in particular that the existence of such a holomorphic function is a sufficient condition to close the period of a polynomial Killing field with any initial potential.
\begin{proposition}
The stabilizer $\Gamma_{\xi_\l}$ in \eqref{eq:stabilizer} depends only on the orbit of $\xi_\l$.
If $\gamma \in \Gamma_{\xi_\l}$ satisfies $F_\l (\gamma)=\pm \un$ for some $\xi_\l \in {\mathcal{P}}_g$ then the same is true for every element in the orbit of $\xi_\l$. The period $\tau$ is related to $t=(\tau,\,0\,,...,\,0) \in \Gamma_{\xi_\l}$.
\end{proposition}
\begin{proof}
This follows from the commuting property of the isospectral action, since $$\pi(\gamma) \pi(t) \,\xi_\l =\pi(t)\pi(\gamma)\,\xi_\l=\pi(t)\,\xi_\l \,.$$
\end{proof}
\begin{proposition}
\label{periodique}
Assume $\xi_\l \in {\mathcal{P}}_g$ has no roots. Then $\gamma \in \Gamma_{\xi_\l}$ if and only if there exists on $\Sigma$ a function $\mu$ which satisfies the following properties:
\begin{enumerate}
\item $\mu$ is holomorphic on $\Sigma^*$ and there exist holomorphic functions $f, g$ on $\C^\times$ with $\mu =f\,\nu +g$.
\item $\sigma ^* \mu = \mu ^{-1},\,\varrho ^* \mu = \bar \mu ^{-1},\, \eta ^* \mu =\bar \mu ^{}$ and $\mu=\pm 1$ at branch points of $\Sigma$.
\item $d \ln \mu$ is a meromorphic 1-form with $d \ln \mu -d(\sum_{i=0}^{g-1} \gamma_i \l^{-i}\nu)$ holomorphic in a neighborhood
of $\l=0$ and $d \ln \mu + d(\sum_{i=0}^{g-1} \bar \gamma_i \l^{i+1-g}\nu)$ is holomorphic at $\l=\infty$.
\end{enumerate}
\end{proposition}
\begin{proof}
For $\gamma \in \R^g$, we write $\exp (\sum_{i=0}^{g-1} \gamma_i \l^{-i} \xi_\l) =F_\l (\gamma) B_\l (\gamma)$ for the Iwasawa decomposition. Then $\gamma \in \Gamma_{\xi_\l}$ if and only if $[F_\l (\gamma),\,\xi_\l]=[B_\l (\gamma),\, \xi_\l] = 0$. Hence $F_\l (\gamma)$ and $B_\l (\gamma)$ act trivially on $\xi_\l$ and $\xi_\l= \zeta_\l (0)=\zeta_\l (\gamma) =F_\l^{-1}(\gamma)\,\xi_\l \, F_\l (\gamma)=B_\l (\gamma) \,\xi_\l\, B_\l^{-1}(\gamma)$. Since $F_\l (\gamma)$ and $B_\l (\gamma)$ commute with $\xi_\l$,
we have by \eqref{eq:crochet} that $F_\l (\gamma)=f(\l)\,\xi_\l + \tfrac{1}{2} {\rm tr}(F_\l (\gamma))\,\un$ and
$B_\l (\gamma) = e(\l)\, \xi_\l + \tfrac{1}{2} {\rm tr}(B_\l (\gamma))\,\un$ with functions $e\,,f$ depending only on $\l$. In this case we define on the spectral curve $\Sigma$ the eigenvalue function $\mu(\l)$ of $F_\l (\gamma)$ by
$$
\mu (\l)=f(\l)\,\nu + g(\l) = f(\l)\,\nu +\tfrac{1}{2} {\rm tr}(F_\l (\gamma)).
$$
We prove that $\mu$ satisfies properties (1),\,(2) and (3). Since $\l \mapsto F _\l (\gamma)$ is holomorphic, the function $\l \mapsto {\rm tr}(F_\l (\gamma))$ is holomorphic on $\Sigma^*$. Since $F_\l (\gamma)$ and $\xi_\l$ commute, can write $f \,\xi_\l = F_\l (\gamma) - \tfrac{1}{2} {\rm tr}(F_\l (\gamma))\un)$, and since $\xi_\l$ has no zeroes conclude that $\l \mapsto f(\l)$ is holomorphic. Hence $\mu: \tilde\Sigma \to \C$ is holomorphic and $f,\,g$ have no poles on $\C^\times$, proving (1).

Property (2) follows since $\mu$ is the eigenvalue of $F_\l (\gamma)$ and $F_\l$ satisfies \eqref{eq:reality-F}.

At $\l=0$, the eigenvalue of the matrix $B_\l (\gamma)=e(\l) \xi_\l + \frac12 {\rm trace}(B_\l (\gamma)) \un$ is holomorphic, so $e(0)=0$ ($\xi_\l$ has a pole at $\l=0$). Hence $B_0 (\gamma)=\un$. Since
$$
\exp (\sum_{i=0}^{g-1} \gamma_i \l^{-i} \nu)
$$
is the product of eigenvalues of the product $F_\l (\gamma)B_\l (\gamma)$, then at $\l=0$, this is precisely the value of $\mu$, which proves (3). Similarly, using $\varrho^* d \ln \mu = \overline{ d \ln \mu}$ allows to deal with the point $\l= \infty$.

Conversely, assume $\mu : \Sigma \to \C$ satisfies conditions (1),(2) and (3). Assume there are two holomorphic function $f,\, g : \C^\times \to \C$ with $\mu=f\nu+g$. Hence $F =f \,\xi_\l +g \,\un$ is holomorphic on $\Sigma^*$ and belongs to the first factor of the Iwasawa decomposition (a consequence of (2)). Due to (3), the
matrix
$$
    B_\l=F_\l^{-1} \exp (\sum_{i=0}^{g-1} \gamma_i \l^{-i} \xi_\l)
$$
is holomorphic in a neighborhood of $\l=0$ and $[B_\l ,\,\xi_\l]=0$. Since $B_\l - \tfrac{1}{2}{\rm tr} (B_\l ) \un$ is at $\l=0$ proportional to $\bigl(\begin{smallmatrix} 0 & 1 \\ 0 & 0 \end{smallmatrix} \bigr)$, $B_\l$ belongs to the second factor of the Iwasawa decomposition.
This shows that $F_\l B_\l$ is the Iwasawa decomposition of
$$
    \exp (\sum_{i=0}^{g-1} \gamma_i \l^{-i} \xi_\l)=F_\l B_\l.
$$
Since $[F_\l (\gamma),\, \xi_\l ]=0$, we conclude that $\gamma \in\Gamma_{\xi_\l}$.
\end{proof}
\begin{remark} \label{muholo}
Holomorphic functions $f,\,g :\C^\times \to \C$ are given by
$$
    \mu = \tfrac{1}{2} \bigl( \tfrac{\mu-\sigma^*\mu}{\nu} \bigr)\,\nu +  \tfrac{1}{2} \left(\mu+\sigma^* \mu \right) =
    f\,\nu + g.
$$
Theorem 8.2 in Forster \cite{Forster-RS} assures that $f,\,g$ extend to holomorphic functions on $\C^\times$. At fixed points of the involution $\sigma$ (zeroes of $\nu$), the function $\mu-\sigma^* \mu$ has zeroes.

In case $\nu$ has higher order roots, the function $(\mu-\sigma^*\mu)/\nu$ may have a pole. Then the condition that $\mu=f \nu +g$ with $f,\,g$ holomorphic is stronger than $\mu$ being holomorphic on $\Sigma^*$.

The 1-form  $d  \ln \mu $  is meromorphic on $\Sigma$, and changes sign under the hyperelliptic involution.
\end{remark}
\subsection*{The closing conditions for regular spectral curves.} In the following we restrict to the case where $a$ has only simple roots. Then the spectral curve $\Sigma$ is a hyperelliptic curve without singularities. The closing condition is simplest as we only need to check the existence of the holomorphic $\nu$ by Remark \ref{muholo}.
\begin{proposition} \label{perioddiff}
Let $a\in\C^{2g}[\l]$ satisfy \eqref{eq:real} and \eqref{eq:hermit}. Then on $\Sigma$ there exist for all $\tau \in \C^\times$ a unique meromorphic differential $\phi$ such that
\begin{enumerate}
\item $\sigma^* \phi = -\phi$, $\varrho^* = - \bar \phi$, $\eta ^*\phi = \bar \phi$.
\item $\phi$ has second order poles at $0$ and $\infty$ without residues, and no other poles.
\item $\phi=\frac{1}{4} \mi\,\tau e^{\mi\Theta/2} d \sqrt{\l}^{-1} + O(1)$ at $\l =0$ and $\phi=\frac{1}{4} \mi\,\bar \tau e^{-\mi\Theta/2} d \sqrt{\l} + O(1)$ at $\l = \infty$.
\item $\int_{\alpha_i}^{1/ \bar \a_i } \phi =   {\rm Re}\left( \int_{\alpha_i}^{1/\bar \a_i}\phi \right) =0$ for all roots $\a_i$ of $a$ where the integral is computed
along the line segment $[\a_i,\,1/ \bar \alpha_i]$.
\end{enumerate}
In conclusion there exists a unique $b\in\C^{g+1}[\l]$ with $\phi = \frac{b\,d\l}{\nu\,\l^2}$ which satisfies $\l ^{g+1} \bar  b (\bar \l ^{-1})=- b (\l)$.
\end{proposition}
\begin{proof}
We make the Ansatz that $\phi = \frac{b \,d\lambda}{\nu\,\lambda^2}$.
The conditions (1), (2) and (3) fix the highest and lowest coefficient of $b$, so there
remain $g$ real free coefficients of $b$. These coefficients correspond to holomorphic differentials satisfying condition (4) ${\rm Re}( \int_{\alpha_i}^{1/\bar \a_i}\phi )=0$. The first equality comes from the reality condition (1) and reads
$$
    \int_{\alpha_i}^{\alpha_i/ | \a_i |} \phi  =
    \overline{\int^{1/ \bar \alpha_i}_{\alpha_i/ | \a_i |} \phi }\,.
$$
Concerning uniqueness, holomorphic differentials whose integrals along cycles
are imaginary, are zero by Riemann's bilinear relations.
\end{proof}
\begin{definition} \label{sigma-hat}
We define a compact Riemann surface with boundary by $\hat\Sigma = \Sigma \setminus \cup \gamma_i$ where $\gamma_i$ are closed cycles over the straight lines connecting the branch points $\alpha_i$ and $1/ \bar \alpha_i$.
\end{definition}
\begin{corollary} \label{periodbonn}
Let $a\in\C^{2g}[\l]$ satisfy \eqref{eq:real} and \eqref{eq:hermit}.
\begin{enumerate}
\item If there is $\tau \in \C$ and $b$ which satisfies {\rm (1)-(4)} of Proposition \ref{perioddiff}, then there exists a unique meromorphic function $h(\l):\hat\Sigma \to \C$ such that $\sigma ^* h(\l)= -h(\l)$ and $dh=  \tfrac{b \,d\lambda}{\nu\, \lambda^2}$.
\item $I(a)$ corresponds to minimal annuli in $\SY ^2 \times \R$ if and only if there exists $\tau \in \C^\times$ with $\tau e^{\mi\Theta/2} \in \R^\times$, and such that the polynomial $b$ defining the function $h(\l):\hat\Sigma \to \C$, satisfies $\sigma ^* h(\l)= -h(\l)$  and  $dh=  \tfrac{b \,d\lambda}{\nu \lambda^2}$. This function continuously extends to boundary segments connecting $\a_i$ and $1/ \bar \alpha_i$ and then takes values on $\mi\pi \Z$ at all roots of $(\l -1)\,a(\l)$.
    \end{enumerate}
\end{corollary}
\begin{proof}
(1) In a small neighborhood of $\Sigma$ over $\l=0$, the function $h$ is uniquely determined by $dh =  \frac{b \,d\lambda}{\nu\,\lambda^2}$ up to some additive constant. This constant is determined by $\sigma ^* h(\l)= -h(\l)$ in this small neighborhood. By conditions on $b$ we have $\int_{\gamma_i} dh =0$, so
$$
    \int_{\gamma_i}  \frac{b \,d\lambda}{\nu \,\lambda^2}=
    2 \int _{\alpha_i}^{1/ \bar \alpha_i} \frac{b\,d\lambda}{\nu\, \lambda^2}=0\,.
$$
Now we can uniquely extend the function $h$ to $\hat\Sigma$.

(2) For an immersed annulus $$X_1 (z)=(F_1 (z)\,\sigma_3 \, F^{-1}_1(z),\,{\rm Re}\,(-\mi e^{\mi\Theta/2}z))$$ in $\SY ^2 \times \R$ the extended frame $F_\l (z)$ admits a period $\tau\in\C^\times$ with $\tau e^{\mi\Theta/2} \in \R$, and periodic Killing field $\zeta_\l (z+\tau)=\zeta_\l (z)$. This implies that  $d \ln \mu$ is a meromorphic differential with second order pole at $0$ and $\infty$ without residues and no other poles and satisfies condition (1)-(3) of Proposition \ref{perioddiff}. The integrals of $d \ln \mu $ along closed cycles are integer multiples of $2\pi\mi$, since the function $\mu$ is globally single-valued by condition (4). Since the extended frame $F_1 (z)$ is periodic then $\mu (1)=\pm 1$. Since $\mu^2 =1$ at branch points $\alpha_i$, it is equivalent to the condition $\ln \mu (\a_i) \in \mi\pi \Z$. Hence there is a polynomial $b$ associated to $\tau \in \C$ such that $d \ln \mu =\frac{b\,d\l}{\nu \l ^2}$.

Conversely, consider $(a,\,b)$ such that on the spectral curve $\Sigma$ the meromorphic
differential $\phi =\frac{b \,d\l}{\nu \,\l ^2} $ satisfies $\int_1^{\alpha_i} \phi  \in \pi\mi \Z$ and $\int_{\alpha_i}^{\alpha_j} \phi  \in  \pi\mi\Z$. While $h(\l)=\int_1^\l \phi$ is a multiple-valued function on $\Sigma$, the function $e^h: \Sigma\to\C$ is again holomorphic. It is described in Proposition~\ref{periodique}
where $\tau$ is given by the residue of $\phi$ at $0$ and $\infty$. Then $\zeta_\l (z+ \tau) = \zeta_\l (z)$, and we can integrate the extended frame $F_\l$. There remains to prove that $F_1(\tau) = \pm \un$. We remark that a solution of the characteristic equation of a solution of $F^{-1}_\l dF_\l = \a(\zeta_\l)$ defines a meromorphic differential $d \ln \mu$ which satisfies (1)--(3) of Proposition \ref{perioddiff}.
By uniqueness of such differentials, $d \ln \mu =\phi$ and $\mu (1)= e^{h(1)} = \pm 1$ so that $F_1$ is $\tau$-periodic.
\end{proof}
\begin{definition}
\label{spectraldata}
The spectral data of a minimal cylinder of finite type in $\SY^2 \times \R$ with sym point at $\l=1$ is a pair $(a,\,b)\in \C^{2g}[\l] \times \C^{g+1}[\l]$ such that
\begin{enumerate}
\item[(i)] $\l ^{2g}  \overline{ a (\bar \l ^{-1})}=a(\l) $, $\l^{-g} a (\l) \leq 0$ for all $\l \in S^1$ and $a(0)=-\frac{1}{16}e^{\mi\Theta}$,
\item[(ii)] $\l ^{g+1} \overline{ b (\bar \l ^{-1})} =- b (\l)$,
\item[(iii)]$b(0) =\tfrac{\,\tau  e^{\mi\Theta}}{32} \in e^{\mi\Theta/2}\R$ (closing condition of the third coordinate).
\item[(iv)] ${\rm Re}\left( \int_{\alpha_i}^{1/\bar \a_i}\frac{b \,d\l}{\nu \l^2} \right) =0$ for all roots $\a_i$ of $a$ where the integral is computed along the straight segment $[\a_i, 1/ \bar \alpha_i]$.
 \item[(v)] The unique function $h:\tilde  \Sigma \to \C$, where $\tilde \Sigma= \Sigma \backslash \cup \gamma_i$ and $\gamma_i$ are closed cycles over the straight lines connecting $\alpha_i$ and $1/ \bar \alpha_i$, satisfies
$$
    \sigma ^* h(\l)= -h(\l) \quad \hbox{ and } \quad dh=  \frac{b\,\, d\lambda}{\nu \lambda^2}\,.
$$
This function continuously extends to boundary segments connecting $\a_i$ and $1/ \bar \alpha_i$ and then
takes values on $\mi\pi \Z$ at all roots of $(\l -1)\,a(\l)$.
 \item[(vi)] When $a$ has higher order roots then $e^h = f\,\nu +g$ for holomorphic $f,\, g:\C^\times \to \C$ with $f(1)=0$.
\end{enumerate}
\end{definition}
%


\section{Bubbletons}

The term 'bubbleton' is due to Sterling-Wente \cite{SteW:bub}. They are the solitons of the theory, and finite type solutions of the sinh-Gordon equations with bubbletons have singular spectral curves. For more details on the relationship between bubbletons, Bianchi-B\"acklund transformations, simple factors and {\sc{cmc}} surfaces we refer to \cite{KilSS, Kil:bub, Kob:bub} and the references therein. By Proposition~\ref{th:pkf-reduce} we can discard roots of a potential $\xi_\l$ without changing the surface, but it remains to discuss higher order roots of $\det \xi_\l$ where $\xi_\l$ is not semi-simple.

We now consider potentials $\xi_\l$ for which the polynomial $a(\l) = -\l \det \xi_\l$ has higher order roots. Roots of $\xi_\l$ come in symmetric pairs $\alpha_0,\,1/\bar{\alpha}_0 \in \C^\times$ and we set in this section $\delta=-\bar \alpha_0 / |\alpha_0|$. Because the polynomial $\l \mapsto a(\l)$ is homogeneous of degree 2, such roots generate even roots of order at least 2 in the polynomial $a$. Then there exists a polynomial $\tilde a$ with
$$
    a(\l)=(\l - \alpha_0)^2 (1-\bar{\alpha}_0 \l)^2 \tilde{a}(\l).
$$
We study such $\xi_\l$ and show that they can be factored into a product of {\emph{simple factors}} (see below) and a potential $\tilde\xi_\l \in I(\tilde a)$. If $\xi_{\alpha_0}\neq 0$ but $\mathrm{ord}_{\a_0}\det \xi_\l \geq 2$, then the matrix $\xi_{\alpha_0}$ is nilpotent and defines a complex line
$$
    L = \ker \xi_{\alpha_0} = {\rm im}\; \xi_{\alpha_0}  \in \C\P^1.
$$
Let $v_1 \in \C^2$ be a unit vector which spans $L$, and complement $v_1$ to an orthonormal basis  $(v_1,\,v_2)$ of $\C^2$. Set
$$
    p(\l)=(\l - \alpha_0)(1-\bar \alpha _0 \l)
$$
and let $Q_L \in \SU$ with $Q_L ( e_1)=v_1$. The entries of $\xi_\l$ in the basis $(v_1,\,v_2)$ are given by
\begin{equation} \label{eq:Q-xi-Q}
    Q_L^{-1}\xi (\l)\,Q_L : = \begin{pmatrix} \tilde \alpha (\l)p(\l) &\tilde  \beta(\l)(1-\bar \alpha_0 \l)^2
    \\ \tilde \gamma (\l)(\l - \alpha_0)^2  & - \tilde \alpha (\l)p(\l)\end{pmatrix}
\end{equation}
where $\l\tilde \a(\l)^2 - \l\tilde \beta(\l) \tilde \gamma (\l) = \tilde a (\l)$. If $(u_1,u_2) \in \ker \xi_{\alpha_0}$ then $(-\bar u_2,\, \bar u_1) \in \ker \xi_{1/\bar \alpha_0}$.

\subsection*{Simple factors.} Define  for $\l \neq \alpha_0, \bar \alpha _0 ^{-1}$ the $\mathrm{SL}_2 (\C)$-valued maps
$$
    \pi _{\alpha_0} (\lambda)  :=  \begin{pmatrix}
    \sqrt{\frac{\lambda-\alpha_0}{1-\bar{\alpha}_0\,\lambda}} & 0 \\
    0 & \sqrt{\frac{1-\bar{\alpha}_0\,\lambda}{\lambda-\alpha_0}}
        \end{pmatrix} \,, \quad \mbox{and} \quad \pi_L := Q_L \pi_{\alpha_0}  Q_L^{-1}\,.
$$
At $\l=0$, we apply the $QR$-decomposition to get
$$\pi_L(0)=Q_L \pi_{\alpha_0} (0) Q_L^{-1}=Q_{0,L} R_{0,L}=R_{1,L}Q_{1,L} $$
where $Q_{0,L},Q_{1,L}  \in \SU$ and $R_{0,L}, R_{1,L}$ are upper triangular matrices of the form $\bigl( \begin{smallmatrix} \rho & r \cr 0 & \rho^{-1} \end{smallmatrix} \bigr)$ for
$\rho \in \R^+$ both depending on $\alpha_0$ and $L$.


\begin{definition} Let $L \in \C\P^1$ and $\alpha_0 \in \C$ with $r<\min\{ |\alpha_0|,\,1/|\a_0| \}$. A {\bf  left  simple factor} is a  map
$$
    g_{L,\alpha_0}(\l)=  \pi^{-1}_{L} (\l)Q_{0,L}\,.
$$
Then $H^{\ell}_{\alpha_0}=\{ g_{L,\alpha_0} \mid L \in \C\P^1\}  \cong \C\P^1$, and $H^{\ell}_{\alpha_0} \subset \Lambda^+_r \SL (\C)$.

A {\bf right simple factor} is a map
$$
    h_{L,\alpha_0}(\l)=        Q_{1,L} \pi^{-1}_{L} (\l)\,.
$$
Then $H^{r}_{\alpha_0}=\{ h_{L,\alpha_0} \mid L \in \C\P^1\} \cong \C\P^1$, and $H^{r}_{\alpha_0} \subset \Lambda^+_r \SL (\C)$.
\end{definition}
\begin{lemma}
There is a one-to-one correspondence between $H^{\ell}_{\alpha_0}$ and $H^{r}_{\alpha_0}$ via
\begin{equation} \label{eq:duality}
    g_{L,\alpha_0}  = \pi^{-1}_{L} Q_{0,L}= Q_{0,L} \pi_{L'}^{-1} =h_{L',\alpha_0}
\end{equation}
where
\[
L'=Q_{0,L}^{-1} L \qquad \hbox{ and } \qquad Q_{1,L'}=Q_{0,L}\,.
\]
Conversely, if $L'$ is given then we can define a QR-decomposition  $\pi_{L'}=R_{1,L'}Q_{1,L'}$ with
$$L=Q_{1,L'} L'.$$
\end{lemma}
\begin{proof}
For any $A \in \SU$ and $L'=AL$ we have $Q_{AL}=AQ_{L}$. Hence $\pi_{L'}=Q_{AL} \pi_{\alpha_0} Q^{-1}_{AL}=A \pi_L A^{-1}$.
Now we define $L'=Q_{0,L}^{-1}L$ and we have $g^{-1}_{L,\alpha_0} = Q_{0,L}^{-1} \pi_{L}=\pi_{L'} Q_{0,L}^{-1}$. With $\pi _L(0) = Q_{0,L}R_{0,L}$, we observe that $$\pi_{L'}(0)=Q_{0,L}^{-1} \pi_L (0) Q_{0,L}=R_{0,L}Q_{0,L}=R_{1,L'}Q_{1,L'}$$
and hence $Q_{1,L'}=Q_{0,L}$. Then
$g^{-1}_{L,\alpha_0} (\l)=\pi_{L'} Q_{0,L}^{-1}=\pi_{L'} Q_{1,L'}^{-1}=h^{-1}_{L',\alpha_0}(\l).$
If $L'=Q_{0,L}^{-1}L$, so that $L=Q_{0,L}L'=Q_{1,L'}L'$.
\end{proof}
Now we consider the matrix $\breve \xi _\l= p^{-1}(\l)\,g_{L,\alpha_0}^{-1} \,\xi_\l \,g_{L,\alpha_0}.$
\begin{proposition}
\label{decomposition}
Let $\xi_\l \in I(a)$ with $\xi_{\alpha_0} \neq 0$ and $\mathrm{ord}_{\a_0} \det \xi_\l \geq 2$. Let $\delta= -\bar \alpha_0/| \alpha_0 |$. Then $\xi_\l$ uniquely factorizes as
$$
    \xi_\l= p(\l)\,g_{L,\alpha_0} \,\breve \xi_\l \,g^{-1}_{L,\alpha_0}
$$
with $L= \mathrm{ker}\, \xi_{\alpha_0}$ and $\breve \xi_\l \in \pk (\delta)$ with $\tilde \xi_\l=g(\delta)^{-1} \breve \xi _\l g (\delta) \in I( \tilde a)$ (i.e. $\breve \xi_\l \in I_{\delta} (\tilde a)$).
\end{proposition}
\begin{proof}
The matrix $\xi_{\alpha_0}$ uniquely defines the line $L$. Consider $$\breve \xi _\l= p^{-1}(\l)\,g_{L,\alpha_0}^{-1} \xi_\l g_{L,\alpha_0}\,.$$ We need to prove that $\breve \xi _\l$ has no poles at $\alpha_0, \bar \alpha^{-1}_0$ and $\breve \xi _\l \in I_{\delta}(\tilde a)$. First define
$$
    Q_L^{-1} \xi_\l Q_L  : = \begin{pmatrix} \tilde \alpha (\l)p(\l) &\tilde  \beta(\l)(1-\bar \alpha_0 \l) ^2 \cr \tilde \gamma (\l)(\l - \alpha_0)^2 & - \tilde \alpha (\l)p(\l)\end{pmatrix}\,.
$$
The following matrix has neither pole nor zero at $\l= \alpha_0,\, \bar \alpha^{-1}_0$:
$$
    Q_L^{-1}Q_{0,L} \breve \xi_\l Q_{0,L} ^{-1}Q_L=\frac{\pi_{\a_0}}{\sqrt{p}} Q_L^{-1} \xi_\l Q_L\frac{\pi_{\a_0}^{-1}}{\sqrt{p}}=  \begin{pmatrix} \tilde \alpha (\l) &\tilde  \beta(\l)  \cr \tilde \gamma (\l) & - \tilde \alpha (\l)\end{pmatrix}\,.
$$
The residue of $\breve \xi_\l$ at $\l = 0$ is $R_{0,L} \delta \hat \xi _{-1} R^{-1}_{0,L}$. As $R_{0,L} \in \Lambda^+ \SL (\C)$, we conclude that the residue takes values in $\mi \delta \R^+\, \bigl( \begin{smallmatrix} 0 & 1 \\ 0 & 0 \end{smallmatrix} \bigr)$. Therefore $\breve \xi_\l \in \calP_{g-2} (\delta)$.
By remark \ref{IN2} on Isometric normalization we conclude that $\tilde \xi_\l=g(\delta)^{-1} \breve \xi _\l g (\delta) \in I( \tilde a)$.
\end{proof}
Using the relation \eqref{eq:duality} between left and right factors immediately yields
\begin{corollary}\label{th:decomp}
Let $\xi_\l \in I(a)$ with $\xi_{\alpha_0} \neq 0$ and $\mathrm{ord}_{\a_0} a(\l) \geq 2$. Then $\xi_\l$ uniquely factorizes as
$$
    \xi_\l= p(\l)\,h_{L',\alpha_0} \,\breve \xi_\l \,h^{-1}_{L',\alpha_0}
$$
with $(L',\breve \xi_\l ) \in \C\P^1 \times  I_{\delta}(\tilde a)$ where  $L'=Q_{0,L}^{-1} L$ and $L= \mathrm{ker}\, \xi_{\alpha_0}$.
\end{corollary}
\begin{remark}
\label{isomnorm}
The factorizations of Proposition~\ref{decomposition} and Corollary~\ref{th:decomp} give rise to pairs in $(L',\,\breve \xi_\l ) \in \C\P^1 \times I_{\delta}(\tilde a)$, and we say that we {\emph{decompose}} such $\xi_\l$ into $(L,\,\breve \xi_\l )$ or $(L',\,\breve \xi_\l )$ or
into $(\tilde L,\,\tilde \xi_\l )$ or $(\tilde L',\,\tilde \xi_\l )$ where $\tilde \xi_\l = g(\delta)^{-1} \breve \xi_\l \;g(\delta) \in I(\tilde a)$ and
$\tilde L = g(\delta)^{-1} L,\tilde L' = g(\delta)^{-1} L'$ depending on the context.
\end{remark}
A special situation occurs when $L'^{\perp}$ is an eigenline of $ \breve \xi_{\alpha_0}$.
\begin{proposition}
Suppose $\xi_\l$ decomposes into $(L',\breve \xi_\l )$ and $\breve \xi_{\alpha_0}\,L'^{\perp} = L'^{\perp}$. Then $\xi_\l$ has zeroes at $\l=\alpha_0,\,1/\bar {\alpha}_0$. Furthermore the singularity of the spectral curve is removable and up to a conformal change of coordinate the potentials $\xi_\l$ and $ \breve \xi_\l$ induce the same extended frame $\breve F_\l (p(0)z)=F_\l(z)$.
\end{proposition}
\begin{proof}
We prove that
$$\dim_{\C} \ker \xi_{\alpha_0}=\dim_{\C} \ker  p(\alpha_0) \pi_{L'}^{-1} (\alpha_0)  \breve \xi_{\alpha_0}  \pi_{L'} (\alpha_0)=2\,.$$
If $L'^{\perp}=<v_2>$ is an eigenline of $\breve \xi_{\alpha_0}$ then $ \pi_{L'}(\l) v_2 =\sqrt{\frac{1-\bar \alpha_0 \l}{\l -\alpha_0}} v_2$ and
$$\pi_{L'}^{-1} (\alpha_0)  \breve \xi_{\alpha_0}  \pi_{L'} (\alpha_0) v_2 = \mu v_2 \hbox{ and } \xi_{\alpha_0} v_2=0\,.$$
If $L'=<v_1>$, we have $ \pi_{L'}(\l) v_1=\sqrt{\tfrac{\l -\alpha_0}{1-\bar \alpha_0 \l}} v_1$ and
$$p(\alpha_0) \pi_{L'}^{-1} (\alpha_0)  \breve \xi_{\alpha_0}  \pi_{L'} (\alpha_0) v_1 =0\,.$$
Then $\xi_\l$ has a zero at $\l=\alpha_0$ and we can remove it without changing the extended frame by Proposition~\ref{th:pkf-reduce}.
\end{proof}
\subsection*{Terng-Uhlenbeck formula.} Let $\xi_\l \in \pk$ with $\xi_{\alpha_0} \neq 0\,,\,\mathrm{ord}_{\alpha_0} a(\l) \geq 2$ for some $\alpha_0 \in \C^\times \setminus \bbS^1$. Suppose $\xi_\l$ decomposes into $(L',\,\breve\xi_\l)$, and let $0 < r < \min\{ |\alpha_0,\,1/|\alpha_0| \}$. Now consider the unitary factor $F_\l: \R^2 \rightarrow \Lambda_r \SU (\C)$ of the r-Iwasawa decomposition
$$
    \exp (z\,\xi_{\l})=F_\l \,B_\l
$$
and define $\breve F_\l: \R^2 \rightarrow \Lambda_r \SU (\C)$ to be the unitary factor of the r-Iwasawa decomposition
$$
    \exp (z \,p(\l) \,\breve \xi_{\l})=\breve F_\l \, \breve B _\l\,.
$$
Terng-Uhlenbeck \cite{TerU} obtained a relationship between $F_\l$ and $\breve{F}_\l$ and found
$$
    F_\l (z)= h_{L',\alpha_0} \breve F_\l (z) h^{-1}_{L'(z),\alpha_0}     \hbox{ with } L' (z)={ {}^t \overline{\breve F} }_{\alpha_0}(z) L'\,.
$$
We provide a proof of this in the appendix. We next show that closing conditions are preserved when changing the first factor in the factorization $(L',\,\breve \xi_\l)$.
\begin{proposition}
If there is $(L'_0,\,\breve \xi_\l) \in \C\P^1 \times I_{\delta}(\tilde a)$, such that $$\xi_{0,\l} =p(\l)\,h_{L'_0,\alpha_0}\,\breve \xi _\l \,h^{-1}_{L'_0,\alpha_0}$$ induces a minimal annulus with period $\tau$, then for any $L'_1 \in \C\P^1$, the extended frame associate to
$\xi_{1,\l} =p(\l)\,h_{L'_1,\alpha_0}\,\breve \xi_\l \,h ^{-1}_{L'_1,\alpha_0}$ produces a $\tau$-periodic minimal annulus.
\end{proposition}
\begin{proof}
The annulus induced by $F_{0, \l}$ is periodic at $\l=1$, so $F_{0,1}(z+\tau)=F_{0,1}(z)$. In particular $F_{0,1}(\tau)=\un$ and the solution of sinh-Gordon equation is periodic $\omega_0 (z+\tau)=\omega_0 (z)$. This period condition on $\omega_0$ implies  $\zeta_{0,\l} (z+\tau)=\zeta_{0,\l} (z)$ for any $\l \in \C^\times$ and $z \in \C$, since entries of $\zeta _{0,\l}$ depend only on $\omega_0$ and its derivatives. Now we remark that the associate polynomial Killing field $\zeta_{0,\l} (z) \in I(a)$ decomposes uniquely as $(L'_0(z), \breve \zeta_\l (z)) \in \C\P^1 \times I_{\delta}(\tilde a)$ where $L'_0(z)=Q_{0,L_0(z)}^{-1} L_0(z)$ and $L_0(z) = \ker \zeta_{0,\alpha_0}(z)$. Thus $\zeta_{0,\alpha_0}(z)$  periodic implies that $L_0(z)$ and $L'_0(z)$ are periodic.
The decomposition  of $\zeta_{0,\l} (z)$ is given by
$$\zeta_{0,\l} (z) = p(\l) g_{L_0(z), \alpha_0} \breve \zeta_\l (z) g^{-1}_{L_0(z), \alpha_0}=p(\l) h_{L'_0(z), \alpha_0} \breve \zeta_\l (z) h^{-1}_{L'_0(z), \alpha_0}$$
and $\breve \zeta_\l (z+\tau)=\breve \zeta_\l (z)$ for any $\l \in \C^\times$.  We can recover explicitly this relation by using the decomposition with $L'=Q^{-1}_{0,L}L$ and the formula
$$
    F_{\l}(z)=h_{L'_0,\alpha_0} \breve F_\l (z) h^{-1} _{L'_0(z), \alpha_0}
$$
where $L'_0(z)={}^t\overline{ \breve{F}} _{\alpha_0}(z) L'_0$. Then
\begin{equation*} \begin{split}
    \zeta_{0,\l }(z) &= F_{\l}(z)^{-1} \xi_{0,\l} F_{\l}(z) \\
     &=  h_{L'_0(z), \alpha_0} \breve F^{-1}_\l (z)h_{L'_0,\alpha_0} ^{-1} \xi_{0,\l} h_{L'_0,\alpha_0}\breve F_\l (z) h^{-1}_{L'_0(z),\alpha_0}\,.
\end{split}
\end{equation*}
Since $L_0(z+\tau)=L_0(z)$, we have $L'_0(z+\tau)=L'_0(z)$. Periodicity $F_1(\tau)=\un$ implies $\breve F_{1}(\tau)=\un$, and then there is a $\tau$-periodic extended frame without a bubbleton associate to the polynomial Killing field $\breve \zeta_\l$. Now we consider $\xi_{1,\l}$ associate to $(L'_1, \breve \xi_\l)$ and we prove that its extended frame has the same period $\tau$. To see that it remains to prove that $\breve F_{\alpha_0} (\tau) =\un$ to conclude the periodicity $L_1'(\tau+z)= { {}^t \overline{\breve F} }_{\alpha_0}(\tau) L_1'(z) = L_1'(z)$.

If $L'_0(\tau)=L'_0$, then $L'_0= L'_0 (\tau)= {}^t \bar {\breve F }_{\alpha_0} (\tau) L'_0$ and $L'_0$ is an eigenvalue
of ${}^t \overline{\breve F }_{\alpha_0} (\tau)={\breve F }^{-1}_{1/ \bar \alpha_0}(\tau)$, hence $(L'_0)^{\perp}$ is an eigenline of  ${\breve F }_{\alpha_0}(\tau)$.

Now $\breve \zeta_{\l} (\tau)=\breve \zeta_{\l} (0)$ for all $\l \in \C^\times$ implies $[\breve \xi_{0,\alpha_0}, \breve F_{\alpha_0}(\tau)]=0$. Thus $(L'_0)^{\perp}$ is not an eigenline of $\breve \xi_{0,\alpha_0}$, since the polynomial Killing field
$\xi_{0,\alpha_0}$ would have a zero and we could remove the singularity of the spectral curve.

Recall that if  $A \in \SL (\C)$ and $B \in {\rm sl}_2 (\C)$, then $[A,B]=0$ implies $A=xB+y\un$.
This implies $\breve F_{\alpha_0}(\tau) = x\breve \xi_{0,\alpha_0} +y\un$. Now  $\breve \xi_{0,\alpha_0} L_0^{\perp} \neq L_0^{\perp}$ implies
$x=0$ and thus $\breve F_{\alpha_0}(\tau) = \pm \un$.
\end{proof}
\subsection*{Group action on bubbletons.} We prove that there is a group action which acts transitively on the first factor $\tilde L'$ of the decomposition $\xi_\l$ into $(\tilde L', \tilde \xi_\l)$. For $\xi_\l \in \pk$ and $\beta \in \C$, define
$$
    m(\beta)\,\xi_\l = \left\{
\begin{array}{ll}
\left( \frac{\beta}{\l-\alpha_0} + \frac{\bar \beta \l}{1-\bar \a_0 \l} \right) \l^{\frac{1-g}{2}} \xi_\l\,, \hbox{ when } g=2k+1\,, \\
\left( \frac{\beta}{\l-\alpha_0} + \frac{\bar \beta \l}{1-\bar \a_0 \l} \right) (\l^{\frac{-g}{2}}+ \l^{1-\frac{g}{2}})  \xi _\l\,,\hbox{ when } g=2k\,.
\end{array} \right.
$$
Let $\exp(m(\beta) \xi_\l) = F_\l(\beta)\,B_\l(\beta)$ be the $r$-Iwasawa factorization for $r < \min\{ |\a_0|,\,1/|\a_0| \}$. We have a complex 1-dimensional isospectral group action $\tilde\pi:\C \times I(a) \to I(a)$ defined by
\begin{equation} \label{eq:1-action}
    \xi_\l (\beta) := \tilde\pi(\beta)\,\xi_\l = F_\l(\beta)^{-1}\,\xi_\l\,F_\l(\beta) = B_\l(\beta) \,\xi_\l \,B_\l^{-1} (\beta)\,.
\end{equation}
The potential $\xi_\l (\beta)$ has a decomposition
$(\tilde L'(\beta), \tilde \xi_\l (\beta))$, and we prove that $\tilde \xi_\l (\beta)=\tilde \xi_\l$ is invariant under this action.
\begin{theorem}
\label{bubbleaction}
Suppose a potential $\xi_\l$ decomposes into $(\tilde L', \tilde \xi_\l) \in \C \P^1 \times I(\tilde a)$. Then the action \eqref{eq:1-action} on $\xi_\l$  preserves the second term $\tilde \xi_\l$. If $\det \tilde \xi_{\alpha_0} \neq 0$, then $\C$ acts on $\tilde L' \in \C \P^1 \backslash \{\tilde L'_1,\tilde L'_2\}$ transitively where  $(\tilde L_1')^\perp, (\tilde L_2')^\perp$ are eigenlines of $\tilde \xi_{\alpha_0}$ and fixed points of the action. If $\det \tilde \xi_{\alpha_0} = 0$, then $\C$ acts on $\tilde L' \in \C \P^1 \backslash \{\tilde L'_3\}$ transitively where  $(\tilde L_3')^\perp$ is the eigenline of $\tilde\xi_{\alpha_0}$.
\end{theorem}
\begin{proof}
Using remark \ref{isomnorm}, we prove the theorem with $L'=g(\delta)\tilde L'$ and $\breve \xi_\l= g(\delta) \tilde \xi_\l g(\delta)^{-1}$. Now
$\tilde L'$ is an eigenline of $\tilde \xi_{\alpha_0}$ if and only if $L'$ is an eigenline of $\breve \xi_\l$. We consider for $\beta \in \C$ the map
$$
     \psi (\l, \beta) =\left\{
    \begin{array}{ll}
    \left( \frac{\beta}{\l-\alpha_0} + \frac{\bar \beta \l}{1-\bar \a_0 \l} \right) \l^{\frac{1-g}{2}}p(\l) \,\breve \xi_\l\,, \hbox{ when } g=2k+1\,, \\
    \left( \frac{\beta}{\l-\alpha_0} + \frac{\bar \beta \l}{1-\bar \a_0 \l} \right) (\l^{\frac{-g}{2}}+ \l^{1-\frac{g}{2}})p(\l) \,\breve \xi_\l\,, \hbox{ when } g=2k\,.
\end{array} \right.
$$
Then $\psi$ satisfies $^t \overline{\psi(1 / \bar \l,\beta)} = - \psi(\l,\beta)$ and $\psi \in \Lambda_r \su$. We have by the r-Iwasawa decomposition $\exp(\psi (\l, \beta))=\breve F_\l (\beta)$. Further, since $\breve B_\l(\beta)=\un$ we have
$$
    \breve F_\l^ {-1} (\beta)\, \breve \xi_\l\, \breve F_\l (\beta) = \breve B_\l(\beta)\, \breve \xi_\l \, \breve B_\l^{-1} (\beta) = \breve \xi_\l\,.
$$
Suppose $\xi_\l (\beta)$ has a decomposition
$(L'(\beta), \breve \xi_\l (\beta))$. We prove next that $\breve \xi_\l (\beta)=\breve \xi_\l$ is invariant by the group action $m(\beta)\xi_\l$.

From $F_\l (\beta) B_\l(\beta)=h_{L',\alpha_0} \exp( \psi (\l, \beta) )h^{-1}_{L',\alpha_0}=h_{L',\alpha_0} \breve F_\l (\beta)h^{-1}_{L',\alpha_0}$, and by the Terng-Uhlenbeck formula we obtain
$$
    F_\l (\beta) = h_{L',\alpha_0} \breve F_\l (\beta)h^{-1}_{L'(\beta) ,\alpha_0} \hbox{ where }L'(\beta)={}^t \overline{ \breve F}_{\alpha_0} (\beta) L'\,.
$$
Applying the action and the invariance of the conjugation of $\breve \xi_\l$ by $\breve F_\l (\beta)$ we have
$$
    \xi_\l (\beta) = \tilde \pi (\beta) \xi_\l= p(\l) h_{L'(\beta) ,\alpha_0}  \breve \xi_\l    h^{-1}_{L'(\beta) ,\alpha_0}\,.
$$
This proves that $\breve \xi_\l$ is invariant while $L'(\beta)$ changes under the group action.  We consider now the map  $\beta \mapsto L'(\beta)$, and prove that this map spans $\C\P^1$ (if $L'(0)\neq L'_0$ and $L'(0) \notin \{ \tilde L_1',\,\tilde L_2',\,\tilde L_3' \}$ ).

First we assume $\det \breve \xi _{\alpha_0} \neq 0$ (so also $\det \breve \xi _{1/\bar  \alpha_0} \neq 0$). We denote by $v_1,\,v_2$ the eigenvectors of $\breve  \xi _{1/ \bar \alpha_0}$, and by $\nu_1,\,\nu_2$ the corresponding eigenvalues. Since ${\rm trace}\; \breve \xi_\l=0$, we have $\nu_2=-\nu_1$.  Now we have for
\[
    \gamma_0 =\left\{
    \begin{array}{ll} \bar{\alpha}_0^{g/2-5/2}(1-|\alpha_0|^2)
     & \mbox{ for $g$ odd } \\
    \bar{\alpha}_0^{g/2-3}(1-|\alpha_0|^2) & \mbox{ for $g$ even}
\end{array} \right.
\]
that
${}^t \overline{\psi (\alpha_0, \beta)}=-\psi (1/ \bar \alpha_0, \beta) = -\bar\beta\,\gamma_0\,\breve \xi_{1/ \bar \alpha_0}$ and ${}^t \overline{\breve F }_{\alpha_0} (\beta) = \exp(- \psi (1/ \bar \alpha_0, \beta))$. With $L'=<xv_1+yv_2>$ we obtain
\begin{align*}
    L'(\beta) &= {}^t \overline{\breve F } _{\alpha_0} (\beta) (xv_1+yv_2) = \exp(- \psi (1/ \bar \alpha_0, \beta))(xv_1+yv_2) \\
        &=\exp(-\bar\beta \gamma_0 \nu_1 )\,x\,v_1  + \exp(-\bar\beta \gamma_0 \nu_2 )yv_2=w\,.
\end{align*}
Hence
$${}^t \overline{\breve F }_{\alpha_0} (\beta)L'=< w> $$
and the map $L' : \C\P^1 \to \C\P^1$ given by $L'(\beta)=<w>$ is surjective since
$$
    \beta:\C\P^1 \to \langle \exp(-\bar\beta \gamma_0 \nu_1 )xv_1+\exp(\bar\beta \gamma_0 \nu_1 )yv_2 \rangle \in \C\P^1
$$
is surjective.

Now assume ${\rm det} \; \breve \xi _{\alpha_0}=\det \breve \xi _{1/\bar  \alpha_0} = 0$. There exists a basis $(v_1,\,v_2)$ with $\breve \xi _{1/\bar  \alpha_0}v_1=0 \hbox{ and }\breve \xi _{1/\bar  \alpha_0} v_2=v_1$, and then
$$
    {\rm exp} (-   \psi (1/\bar\a_0,\,\beta))( x\,v_1 + y\,v_2) = (x - \bar\beta \gamma_0 y)\,v_1 + y\,v_2.
$$
For $y \neq 0$ the map
$
    \beta: \C\P^1 \to  \langle (x - \bar\beta \gamma_0 y)v_1 + yv_2 \rangle \in \C\P^1
$
is surjective. When $y=0$ then $L'=v_1$ is an eigenvector of $\breve \xi_{\alpha_0}$, and
there is no bubbleton.
\end{proof}


\section{Spectral curves of the Riemann family}
The Riemann family consists of embedded minimal annuli in $\bbS^2 \times \R$ that are foliated by horizontal constant curvature curves of $\SY^2$. From \cite{hauswirth2006} these annuli can be conformally parameterized by their third coordinate with $Q=\frac14 (dz)^2$ and the metric $ds^2=\cosh^2 \o \,|dz|^2$ is obtained from real-analytic solutions of the Abresch system \cite{Abr}
\begin{equation} \label{eq:Abresch-system}
\left\{\begin{array}{ll}
    \Delta \o  + \sinh  \o \cosh \o &=0\,, \\
    \o_{xy} - \o_x\,\o_y \,\th &=0\,.
\end{array} \right.
\end{equation}
The second equation is the condition that the curve $x \mapsto (G(x,y),y)$ has constant curvature.
This condition induces a separation of variables of the sinh-Gordon equation,
and solutions can be described by two elliptic functions
\begin{equation} \label{eq:elliptic}
    f(x) = \frac{-\o _x}{\ch } \quad \mbox{ and } \quad  g(y) = \frac{- \o _y}{\ch }
\end{equation}
of real variables $x$ and $y$ respectively, and for $c<0,\,d<0$ solve the system
\begin{equation*}
\begin{split}
- (f_x)^2=f^4 +(1 +c-d )f^2 + c\,,  &\qquad
-f_{xx}= 2f^3 +(1 +c-d )f \,, \\
- (g_y)^2=g^4 +(1 +d-c )g^2 + d\,, &\qquad
-g_{yy}= 2 g^3 +(1 +d-c )g\,.
\end{split}
\end{equation*}
We can then recover the function $\o$ by
\begin{equation} \label{eq:3.1}
\sinh \o=(1 +f^2+g^2)^{-1}(f_x + g_y)\,.
\end{equation}
The spectral curves of members of the Riemann family have spectral genus 0, 1 or 2.
The spectral genus zero case consists of flat annuli $\gamma \times \R$, where $\gamma \subset \SY^2$ is a great circle. In this case the solution of the sinh-Gordon equation is the trivial solution $\o \equiv 0$. The spectral genus 1 case consists of solutions of the sinh-Gordon equation that only depend on one real variable. The corresponding minimal annuli are analogous to associate family members of Delaunay surfaces. In particular the spectral genus 1 case contains the rotational annuli and helicoids, and these are foliated by circles.  Amongst the spectral genus 2 surfaces, the corresponding minimal annuli are again foliated by circles, but no longer have rotational symmetry. This condition endows the spectral curve with an  additional symmetry.
\subsection*{Spectral genus 0.} We first study annuli with spectral curves of genus 0. Inserting the trivial solution of the sinh-Gordon equation $\omega_0 \equiv 0$ into \eqref{eq:symb}, and setting $\gamma =1$, gives
\begin{equation}
	\alpha (\o_0) = \frac{1}{4}\,
	\begin{pmatrix}
    0 &  \mi \lambda^{-1} \\
 \mi &  0
  \end{pmatrix}\, dz + \frac{1}{4}\,\begin{pmatrix}
    0 & \mi  \\
 \mi  \lambda & 0
  \end{pmatrix} d \bar{z}
\end{equation}
The solution of $F_\l^{-1}dF_\l  =\alpha (\o_0),\,F_\l (0) = \un$ is given by
\begin{equation} \label{eq:flat-frame} \begin{split}
    F_\l &= \exp \left( \frac{\mi}{4}\begin{pmatrix} 0 & \lambda^{-1} z + \bar z \\ z + \lambda\,\bar{z} & 0 \end{pmatrix} \right) \\
    &=
    \begin{pmatrix} \cos \bigl(\frac{1}{4}(\frac{z}{\sqrt{\l}} + \bar z \sqrt{\l}) \bigr) &
    \mi \frac{\sin \bigl( \frac{1}{4}(\tfrac{z}{\sqrt{\l}} + \bar z \sqrt{\l}) \bigr)}{\sqrt{\l}} \\
    \mi \sqrt{\l}\sin \bigl( \frac{1}{4}(\tfrac{z}{\sqrt{\l}} + \bar z\,\sqrt{\l}) \bigr)  &
    \cos \bigl( \frac{1}{4}( \tfrac{z}{\sqrt{\l}} + \bar z \sqrt{\l}) \bigr) \end{pmatrix}.
\end{split}
\end{equation}
The horizontal part of $( F_\l\, \sigma_3\, F_\l^{-1},\, {\rm Re}( -\mi\lambda^{-1/2}z))$ computes to
\[
  F_\l\, \sigma_3\, F_\l^{-1} = \begin{pmatrix} \mi \cos \bigl( \mathrm{Re}\,(z\,\lambda^{-1/2}) \bigr) &
    \lambda^{-1/2}  \sin \bigl( \mathrm{Re}\,(z\,\lambda^{-1/2}) \bigr) \\
        -\lambda^{1/2}  \sin \bigl( \mathrm{Re}\, (z\,\lambda^{-1/2}) \bigr) &
            -\mi \cos \bigl( \mathrm{Re}\,(z\,\lambda^{-1/2}) \bigr) \end{pmatrix}\,.
\]
Identifying $\su \cong \R^3,\,\bigl( \begin{smallmatrix} \mi w & u + \mi v \\ -u + \mi v & -\mi w \end{smallmatrix} \bigr) \cong (u,\,v,\,w)$, evaluating the associated family at $\lambda =1$, and writing $z = x + \mi y$, we obtain the conformal minimal immersion $\C \to \bbS^2 \times \R \subset \R^4$ given by
\[
    X_1 (x,\,y) = \bigl( \sin x,\,0,\,\cos x,\,y \bigr)\,.
\]
Restricting $X_1$ to the strip $(x,\,y) \in [0,\,2\pi] \times \R$ then gives an embedded minimal flat annulus in $\bbS^2 \times \R$. Evaluating the associated family at some other point $\lambda_0 \in \bbS^1$, then $\tau$-periodicity requires that $\rm{Re}\,(\tau \l_0 ^{-1/2}) \in 2\,\pi\,\bbZ$.

We next compute the corresponding spectral data $(a,\,b)$. Since
\[
    F_\l = \exp (z\,\xi_\l - \bar z \,\overline{\xi_{1/\bar\l}}^t ) \quad \mbox{ for } \quad
    \xi_\l= \tfrac{\mi}{4} \bigl(\begin{smallmatrix}    0 &  \l ^{-1} \\  1  & 0  \end{smallmatrix} \bigr)
\]
coincides with the extended flat frame \eqref{eq:flat-frame} computed above, we conclude that $\xi_\l$ is a potential for the flat surface. Hence $a(\lambda) = -\l \det \xi_\l = -1/16$, and the spectral curve \eqref{eq:sigma} is the 2-point compactification of $\{ (\nu,\,\l) \mid \nu^2 = -\l^{-1}/16 \}$. The flat annulus has the simplest possible spectral curve. It is a genus zero hyperelliptic curve, so a double cover of $\C\bbP^1$ with two branch points.

The eigenvalues of $F_\l$ in \eqref{eq:flat-frame} are $\exp ( \pm \tfrac{\mi}{4} (z\l^{-1/2} + \bar z \l^{1/2}))$. Therefore the logarithmic eigenvalue (up to sign) of the monodromy with respect to the translation $z \mapsto z + 2\pi$ is
\[
    \ln \mu (\l) = \frac{\pi\mi }{2}\left(\l^{-1/2} + \l^{1/2} \right)\,.
\]
Then
\[
    d\ln \mu = \frac{\pi\mi(\l -1)}{4 \l^{3/2}}\,d\l = \frac{\pi\,(1-\l)}{16 \l^2 \nu}\,d\l
\]
since $\l^{3/2} = -4\mi\l^2\nu$. Thus $b(\l)=\frac{\pi}{16}(1-\l)$.
%
%
\subsection*{ Spectral genus 1.}  We apply the Pinkall-Sterling iteration to the case where $\o$ satisfies $\alpha \o_z + \beta \o_{\bar z}=0$ for $\alpha,\beta \in \C$.
The relation implies that $|\alpha|=|\beta|$ and up to a change of coordinate we assume without loss of generality that $\o_x =0$.
Then $\o_z= - \o_{\bar z}$, and we are exactly in the setting of Abresch's system \cite{Abr} and there is a constant $d <0$ with $-b_y^2= (b^2+1)(b^2 +d)$ where
\begin{equation} \label{ab}
    b(y)=\frac{-\o_y}{\cosh \o} \quad \mbox{ and }
    \quad \sinh \o = \frac{b_y}{1+b^2}\,.
\end{equation}
We now use the Pinkall-Sterling iteration to compute the polynomial Killing field. Starting with $u_{-1}=\sigma_{-1}=0$ and $\tau_{-1}=\mi/4$, and using $4\o_{z\bar z}= -\tfrac{1}{2}\sinh (2\o)$, gives
\[
u_0=-4\mi \o_z \tau_{-1}=\o_z \,, \qquad
\sigma_0 = \gamma\,e^{2\o} \tau_{-1} + 4\mi \gamma\, u_{0;\bar z}= \tfrac{1}{4}\,\mi\gamma\,e^{-2\o}\,.
\]
We use the function $\phi_0$ to compute $\tau_0= 2\mi\bar\gamma ( \tfrac{1}{2}\phi_0 - u_{0;z})$.
We have $u_0 = -\o_{\bar z}$, and then
\begin{equation*} \begin{split}
    \phi_{0;z} &= -4 \o_z \o_{z\bar z}= \tfrac{1}{4}\,(\cosh (2\o))_z \,, \\
    \phi_{0;\bar z} &=-\o_{\bar z}\sinh \o \cosh \o  = \tfrac{1}{4}(\cosh (2\o))_{\bar z}\,.
\end{split}
\end{equation*}
Then $\tau_0= 2\mi\bar\gamma\,( \o_{z \bar z} + \tfrac{1}{8} \cosh (2\o) ) = \tfrac{1}{4}\mi\bar\gamma\, e^{-2\o}$.

At the next step we find $u_1=-2\mi \tau_{o;z} - 4\mi \o_z \tau_{0}=0$, $\sigma_1 = \gamma\,e^{2\o} \tau_{0} = \tfrac{\mi}{4}$ and $\tau_1 =0$. This gives the polynomial Killing field \eqref{potentiel} as
$$
    \zeta_\l = \frac{\mi}{4}\,\begin{pmatrix}
    -2\,\o_y  & e^{\o}\l^{-1} + \bar\gamma\, e^{-\o} \\
    \gamma\, e^{-\o} + e^{\o}\l  & 2\,\o_y
    \end{pmatrix}\,.
$$
Then $a(\l) = - \l \det \zeta_\l = -\tfrac{1}{16}\left( \gamma + (2\cosh(2\o) + 4\o_y^2)\,\lambda + \bar\gamma \lambda^2\right)$. Using \eqref{ab} gives
$$
2 \o_y^2 + \cosh 2\o= 2 b^2 \cosh^2 \o + 1+2 \sinh^2 \o=1-2d\,.
$$
so that $a(\l) =  -\tfrac{1}{16}\left( \gamma + 2(1-2d)\,\lambda + \bar\gamma \lambda^2\right)$. Its roots are $(2d-1 \pm 2\sqrt{d^2-d})\,\gamma$. If $\gamma = \pm 1$,
then this polynomial satisfies the additional symmetry
\begin{equation} \label{eq:additional-symmetry}
    \l^{2g} a(1/\l) =a(\l)\,,
\end{equation}
for $g=1$, and it has two real roots. The corresponding spectral curve is a double cover of $\C\bbP^1$ branched at 4 points, so a hyperelliptic curve of genus 1.
To close the surface, we have to close the third coordinate  with $Q=\tfrac{1}{4}\gamma \l_0^{-1}(dz)^2=\tfrac{1}{4} (dz)^2$. Riemann annuli of spectral genus 1 satisfy $\o_x=0$ or $\o_y=0$. This means that $\l_0=\pm \gamma$. We can parameterize the annulus in such a way that $\l_0= 1$, and apply the iteration with $\gamma= 1$. This corresponds to the case where $y \mapsto \o(y)$ depends only on its third coordinate and describes rotational examples. In the other case where $\l_0=1$ and $\gamma=-1$,
this corresponds to the helicoidal surfaces, where the surface is foliated by horizontal geodesics. In this case the function $\o$ depends only on the variable $x$.


\subsection*{ Symmetric spectral genus 2.} We next consider general real-analytic solutions of Abresch's system. In this case we first prove that they correspond to spectral genus 2 surfaces.
\begin{lemma}
Every solution $\o : \C \to \R$ of Abresch's system \eqref{eq:Abresch-system} satisfies
\begin{equation} \label{eq:3-relation}
    \o_{zzz}-2\o_{z}^3= -\tfrac{1}{4}\,\o_{\bar z} + \tfrac{1}{2}\,(c-d)\,\o_z\,.
\end{equation}
\end{lemma}
\begin{proof}
Differentiating $\o_z = -\tfrac{1}{2}\,(f-\mi g)\,\cosh(\o)$ gives
$$
    \o_{zz} = \tfrac{1}{4}(f- \mi g)^2 \sinh (\o) \cosh (\o)  - \tfrac{1}{2} (f-\mi g)_z \cosh (\o)\,.
$$
Now using the equation of the system, we have
$$
    2(f-\mi g)_z = f_x -g_y = \tfrac{(1+f^2+g^2) (g^2-f^2 + d-c)}{f_x +g_y} =
    \tfrac{g^2-f^2 + d-c}{\sinh (\o)}\,.
$$
Then $\o_{zz} = \tanh (\o) \; \o^2_z - \tfrac{1}{4}\coth (\o) \left(g^2-f^2 + d-c \right)$ and thus
$$
\tanh (\o)\; \o_{zz} - \tanh^{2}(\o) \,\o_z^{2} + \tfrac{1}{4}(g^2-f^2) = \tfrac{1}{4}(c-d) \in \R
$$
since the imaginary part is $\tanh(\o)(\o_{xy} - \tanh(\o)\, \o_x \o_y)=0$.
Using the expression for $f_x+g_y$ and $g_y-f_x$ we obtain
\begin{align*}
2 f_x &= (1+f^2+g^2)\sinh (\o) -(f^2-g^2 +c-d) \sinh ^{-1}(\o)\,, \\
2 g_y &= (1+f^2+g^2)\sinh (\o) +(f^2-g^2 +c-d) \sinh ^{-1}(\o)\,.
\end{align*}
To understand higher order derivative we write
\begin{equation*}
\begin{split}
    -2(g^2 &- f^2 +(d-c))_z = \left( 2 f f_x + 2\mi g g_y\right) \\
    &=(f+\mi g)(1+f^2+g^2) \sinh (\o) - (f -\mi g) (f^2 - g^2 +c-d)\sinh ^{-1} (\o)\,.
\end{split}
\end{equation*}
We can check that
\begin{equation*}
\begin{split}
    \tfrac{1}{8}\,(f+\mi g)\cosh (\o) &= -\tfrac{1}{4} \o_{\bar z}\,, \\
    \tfrac{1}{8}\,(f+\mi g) (f^2+g^2)\cosh (\o) &= \tfrac{\cosh (\o)}{8}(f-\mi g)^3 + \tfrac{\cosh (\o)}{2} \mi fg (f-\mi g) \\ &=-\tfrac{\o^{3}_z }{\cosh^2 (\o)} -\mi fg \o_z\,.
\end{split}
\end{equation*}
Now we compute
\begin{equation*}
\begin{split}
\o_{zzz} - 2\o_z ^3 &= -2 \o^3_z + \o^{3}_z \mathrm{sech}^2 (\o) +2 \tanh (\o) \;\o_z \o_{zz} \\
    &\qquad + \tfrac{1}{4}\,\o_z \mathrm{csch}^2(\o) \left( g^2 -f^2 + d-c \right) \\
    &\qquad  -\tfrac{1}{4} \,\mathrm{coth}(\o) \left( g^2 - f^2 + d-c \right)_z \\
&=-2 \o^3_z + \o^{3}_z \mathrm{sech}^2(\o) - \o^{3}_z\mathrm{sech}^2(\o) - \mi fg \o_z \\
    &\qquad + 2 \o_z \left(\tanh^2 (\o)\; \o^2_z + \tfrac{1}{4}(f^2-g^2 + c-d)\right) \\
&\qquad + \tfrac{1}{4}\, \o_z ( g^2 -f^2 + d-c) \,\mathrm{csch}^2(\o)
-\tfrac{1}{4}\,\o_{\bar z}  \\
&\qquad - \tfrac{1}{8} (f-\mi g)(f^2-g^2 +c-d)\,\mathrm{coth}(\o) \,\mathrm{csch}(\o) \\
&= -\tfrac{1}{4}\,\o_{\bar z} + \o_z \left( 2\tanh^2(\o) \; \o^2_z -2 \o_z^2 \right) \\ &\qquad +
    \tfrac{1}{2}\,(c-d)\,\o_z + \tfrac{1}{2}\,(f -\mi g)^2 \o_z \\
&= -\tfrac{1}{4}\,\o_{\bar z} +  \tfrac{1}{2}\,(c-d)\,\o_z \,.
\end{split}
\end{equation*}
\end{proof}
Next we use the iteration of Pinkall-Sterling to compute the spectral curve associated to the algebraic relation \eqref{eq:3-relation}. It is a priori a one-parameter family of algebraic relations but $\o$ itself encodes other invariant quantities than the one found in the expression of $\o_{zz}$. The Pinkall-Sterling iteration gives
\begin{align*}
u_{-1}&= u_{2} =\sigma_{-1}= \tau_{2} = 0\,,\, u_{0} = \o_z\,,\, u_{1} = \bar\gamma \o_{\bar z} \,, \\ \sigma_{0}&= \tfrac{1}{4} \mi\gamma e^{-2\o}\,,\,  \sigma_1=2\mi e^{-2\omega} \bigl( \o_{zz}+\o_z^2 + \tfrac{1}{4}(c-d)\bigr)\,,\,\sigma_{2} = -\tfrac{1}{4}\,\mi \bar\gamma, \\
\tau_{-1} &= \tfrac{1}{4}\,\mi \,,\, \tau_{0} = 2\mi\bar\gamma \bigl( \o_{z}^2 - \o_{zz} + \tfrac{1}{4}(c-d)\bigr)\,,\,  \tau_1 = -\tfrac{1}{4}\,\mi \gamma^{-2} e^{-2\o}\,.
\end{align*}
This defines a solution of the Lax equation $\Phi$ by \eqref{potentiel} of degree $N=2$. To obtain a polynomial Killing field we skew-symmetrize and define
\[
    \zeta _\l (z) =\frac12  \Phi_\l (z) -  \frac{\lambda}{2}\,\overline{\Phi_{1/ \bar \l} (z)}^t\,.
\]
Then $\lambda\,\overline{\zeta_{1/\bar\lambda}}^t = -\zeta_\lambda$ and has $\lambda^{-1}$-coefficient $\tfrac{1}{8} \mi e^\o (1-\gamma) \bigl( \begin{smallmatrix} 0 & 1 \\ 0 & 0 \end{smallmatrix} \bigr)$. For $\zeta_\l$ to be $\calP_2$-valued we require that $\tfrac{1}{8} \mi e^\o (1-\gamma) \in \mi \R^+$, which means that $\gamma = -1$. In this case, $Q=\frac{1}{4} \gamma \l^{-1} (dz)^2$ implies that the Sym point is at $\l=-1$ and the entries of the polynomial Killing field are
\begin{equation*} \begin{split}
    \alpha &= (\o_z - \l \o_{\bar z })\,, \\
    \beta  &= \tfrac{\mi}{4}\l^{-1}e^{\o} + \mi \bigl( e^{-\o}(\o^2_{\bar z}+ \o_{\bar z \bar z}) -  e^{\o}(\o^2_{z} + \o_{zz}) \\ &\qquad + \tfrac{1}{4}(c-d)(e^{-\o} - e^{\o}) \bigr) - \tfrac{\mi}{4}\l\,e^{-\o}\,, \\
    \gamma &= -\tfrac{\mi}{4}e^{-\o} + \mi \l \bigl( e^{-\o}(\o^2_{ z} + \o_{zz}) - e^{\o}( \o^2_{\bar z} + \o_{\bar z \bar z}) \\ &\qquad + \tfrac{1}{4}(c-d)(e^{-\o} - e^{\o}) \bigr) +\tfrac{\mi}{4} \l^2 e^{\o}\,.
\end{split}
\end{equation*}
To compute the spectral curve, we have only to compute $a(\l) = \l(\alpha^2 + \beta\,\gamma)$ at one point. We choose a point where $\o(x_0,\,y_0)= \partial_x f(x_0)= \partial _y g (y_0)=0$.
At this point $\o_{zz}=\o_{\bar z \bar z}=0$, thus
\begin{align*}
    f(x_0) &= f_0=-\partial_x \o (x_0) = \tfrac{1}{2} \bigl( -1+d-c +\sqrt{\Delta}   \bigr) \,,\\
    g(y_0) &= g_0=-\partial_y \o (y_0) = \tfrac{1}{2} \bigl( -1+c-d +\sqrt{\Delta}   \bigr)\,.
\end{align*}
where $\Delta=(1+c-d)^2-4c=(1+d-c)^2-4d$.

Writing $a(\l)= \l\,(\alpha^2 + \beta\gamma) = a_0 + a_1\l + a_2\l^2 + a_3\l^3 + a_4\l^4$, and using $\o_z(x_0,\,y_0) = -\tfrac{1}{2} (f_0 - \mi g_0)$ and $\o_{\bar z}(x_0,\,y_0) =  -\tfrac{1}{2} (f_0 + \mi g_0)$, a computation gives the real coefficients
\begin{equation*} \begin{split}
a_0 &= a_4 = \tfrac{1}{16},\\
a_1 &= a_3 = \tfrac12 \,(\o_z^2 + \o_{\bar z}^2) =\tfrac14( f_0^2 - g_0^2), \\
a_2 &= \,(\o_z^4 + \o_{\bar z}^4) - 2\,(\o_z \o_{\bar z} + \o_z^2\o_{\bar z}^2) - \tfrac{1}{8}
    = -\tfrac{1}{8} - \tfrac12  g_0^2 - \tfrac12 f_0^2 - f_0^2g_0^2\,.
\end{split}
\end{equation*}
The four real roots of $a(\l)$ are $-1 - 2 f_0^2 \pm 2 \sqrt{f_0^2 + f_0^4},\, 1 + 2 g_0^2 \pm 2 \sqrt{g_0^2 + g_0^4}$, and $a(\l)$ satisfies the additional symmetry \eqref{eq:additional-symmetry} for $g=2$. In summary the spectral data of the 2-parameter family of the Riemann family are given by
\begin{proposition} \label{low-spectral-data}
The genus 0 spectral data of an embedded annulus with Sym point $\l =1$ is given by
\begin{enumerate}
\item[1)] $a(\l)=- \frac{1}{16}$ and $b(\l)=\frac{\pi}{16}(1-\l)$
\end{enumerate}
The genus 1 spectral data of an embedded annulus with Sym point $\l =1$ is given by
\begin{enumerate}
\item[2)] $a(\l)=\frac{1}{16\a}(\l-\alpha)(\a \l -1)$ for $\a \in (0,\,1)$ and $b(\l)= \frac{b(0)}{\gamma} (\l - \gamma)(\gamma \l -1)$, with $\gamma \in (\alpha,\,1)$ and $b(0) \in \mi\R$
both determined by $\a$.
\item[3)] $a(\l)=\frac{-1}{16\beta}(\l+\beta)(\beta \l +1)$ for $\beta \in (0,\,1)$ and $b(\l)= \frac{b(0)}{\gamma} (1- \l)(1+ \l)$ and $b(0) \in \R$ determined by $\beta$.
\end{enumerate}
The genus 2 spectral data of an embedded annulus with Sym point $\l =1$ is given by
\begin{enumerate}
\item[4)] $a(\l)=\frac{1}{16\beta \alpha}(\l-\a)(\a \l -1) (\l +\beta)(\beta \l +1)$ for $\a, \beta \in (0,1)$ and $b(\l)=\frac{b(0)}{\gamma}(1+\l)(\l - \gamma)(\gamma \l -1)$ for $\gamma \in (\alpha,\,1)$ and $b(0) \in \mi \R$ both determined by $\a$ and $\beta$.
\end{enumerate}
In conclusion, the polynomial $a$ satisfies the additional symmetry $\l^{2g} a(1/\l) = a(\l)$ and
\begin{enumerate}
\item[a)] $\l^{g+1} b(1/\l) = b(\l)$ if $a$ has a root $\alpha \in \R^+$ and $b(0) \in \mi \R$;
\item[b)] $\l^{g+1} b(1/\l) =- b(\l)$ if $a$ has only roots in $\R^-$ and $b(0) \in  \R$.
\end{enumerate}
\end{proposition}
\begin{proof}
We have seen above that spectral curves of the Riemann family have an additional involution $(\l,\,\nu) \to (\l^{-1},\, \l^{1-g} \nu)$, since in all cases $\l^{2g} a(1/\l) = a(\l)$.
Now depending on $a(\l)$, we construct a function $h$ which satisfies the closing condition of the annulus. We prove that there are constants $\gamma$ and $b(0)$ such that $b$ satisfies the closing condition of Proposition \ref{perioddiff}.

First we remark that $\l^{g+1} \overline {b (1 / \bar \l)} =-b(\l)$ by construction.
We look for $h$ satisfying $\sigma^*h =-h$ and $dh= \frac{b\,d\l}{\nu \l^2}$.
First we need to prove that $h$ is well defined on $\hat\Sigma$ (Definition \ref{sigma-hat}).

In cases 1) and 3), there is a root $\a \in (0,1)$. Along the segment $(\a, \,1 / \a)$, the polynomial $b(\l) \in \mi\R$ and $\nu \in \mi\R$. Since $b(\l)$ has exactly one root in the interval $(\a,\,1)$ at $\gamma \in (\a,\,1)$, there exists exactly one value of $\gamma$ which cancels the following integral for a given $\a$ and $\beta$.
Using the additional symmetry, there is a real $\gamma \in (\alpha,\,1)$ with
$$
\int_{\a}^{1 / \a} \frac{b}{\nu \l^2}  \,d\l = 2 \int_{\a}^{1} \frac{b}{\nu \l^2} \,d\l=0\,.
$$
Moreover,  in cases 2) and 3), we have by the reality condition that
$$
\int_{-\beta}^{-1 / \beta} \frac{b}{\nu \l^2} \,d\l =0\,.
$$
Now the function $h$ with $dh = \frac{b\,d\l}{\nu \l^2}$ and $\sigma^* h = -h$ is well defined on $\hat\Sigma$, the curve with the two cycles around $(-1/ \beta,\, -\beta)$ and $(\alpha,\,1/ \alpha)$ removed.

For $\kappa :(\l,\,\nu) \to (\bar\l ,\,\bar\nu )$ on $ \tilde \Sigma$ have $\kappa^* h = - \bar h$, since $b(\bar \l)= \overline{ b(\l)}$.

On the real axis between $\l=0$ and $\l=\alpha$ (or $\l=1$ in the case 2)), the polynomial $a$ takes real positive values. Then the segment $(0,\,\a]$ (or $(0,\,1]$ in the case 2) is a set of fixed points for the involution $\kappa$. On this segment we deduce that $h=\kappa^*h$ and the function $h$ is purely imaginary on this segment. Since the integral $\int _\a ^1 d \ln \mu =0$, the function $h$ is imaginary at $\l=1$ and $h(\alpha)=h(1) \in \mi\R$.

The involution $\varrho$ in \eqref{eq:involutions} leaves $\bbS^1$ invariant, and we have $\varrho^* dh= -\overline{dh}$.
Hence $dh \in \mi\R$ on $\bbS^1$. Thus on the unit circle $h$ stays imaginary, so in particular $h \in \mi\R$ at $\lambda=-1$.

The segment $(-1,\,-\beta)$ is a set of fixed points for $\kappa$ and the function $h \in \mi\R$ on this segment. Since on the real line the function $a(\l)$ changes sign and become real negative on $(-\beta,\,0)$, the function $h \in \R$ on this segment. We can then deduce that $h(-\beta) =0$ at this point.
Now we can choose the value of $b(0)$ to get a multiple value of $\pi\mi$ at the sym point $\l =1$. This proves the closing condition and concludes the proof of the proposition.
\end{proof}
\begin{lemma}
If $\gamma$ is a root of $b$, then the corresponding function $|\mu (\gamma)| \neq 1$.
\end{lemma}
\begin{proof}
For $\l \in [\alpha,\,\bar \alpha^{-1}]$, the function $h= \ln \mu$ is real and $\int _{\alpha} ^1 dh =0$. Then $\gamma$ is a root of $dh$ and is contained in $(\alpha,\,1)$. Since ${\rm Re}\,h(\alpha)={\rm Re}\,h(1)=0$, the value $\gamma$ is the local critical point of $h$. Then ${\rm Re}\,h(\gamma) \neq 0$, and thus $|\mu| \neq 1$.
\end{proof}

\appendix
\section{Terng-Uhlenbeck Formula}
\begin{proposition}
Let $h_{L',\alpha_0} \in H^r_{\alpha_0}$ the simple factor with $\alpha_0 \in \C^\times \setminus \bbS^1$, with $r<\min\{|\alpha_0|,\,1/|\a_0|\}$ and $L' \in \C\P^1$. Then
$$
    F_\l (z)= h_{L',\alpha_0} \breve F_\l (z) h^{-1}_{L'(z),\alpha_0}     \hbox{ with } L' (z)={ {}^t \overline{\breve F} }_{\alpha_0}(z) L'\,.
$$
\end{proposition}
\begin{proof} By r-Iwasawa decomposition
\begin{equation*} \begin{split}
    F_\l(z) \,B_\l(z) &= \exp(z \xi_{\l}) = \exp(zp(\l) h_{L' ,\alpha_0}\breve \xi_{\l} h^{-1}_{L',\alpha_0}) \\ &= h_{L',\alpha_0} \exp (z p(\l) \breve \xi_{\l} ) h^{-1}_{L',\alpha_0} \in \Lambda_r \SL (\C)
\end{split}
\end{equation*}
and $h_{L',\alpha_0} \in \Lambda_r ^+ \SL (\C)$ ($h_{L',\a_0}(0)= Q_{1,L'} Q_{L'} \pi_{\alpha_0} ^{-1} (0)Q_{L'}^{-1} =R_{1,L'}^{-1}$ at $\l=0$ and $\pi$ is holomorphic for $r <|\a _0|$).
Then
$$
    \exp (z p(\l) \breve \xi_{\l})h^{-1}_{L',\alpha_0} =
    \breve F_\l (z) \breve B_\l'(z)= \breve F_\l (z) \,\breve B_\l (z)  h^{-1}_{L',\alpha_0}\,.
$$
Now we have
\begin{equation*} \begin{split}
    F_\l(z)B_\l(z) &= h_{L',\alpha_0} \exp (z p (\l) \breve \xi_{\l}) h^{-1}_{L',\alpha_0} = h_{L',\alpha_0} \breve F_\l(z) \breve B_\l'(z) \\ &=(h_{L',\alpha_0} \breve F_\l (z) H)(H^{-1} \breve B_\l' (z))\,.
\end{split}
\end{equation*}
By uniqueness of the r-Iwasawa decomposition we have only to prove that if $H=h^{-1}_{L'(z),\alpha_0} \in H^r_{\alpha_0}$ with $L'(z)={}^t \bar {\breve{ F}}_{\a_0}(z) L'$ then $h_{L',\alpha_0} \breve F_\l (z) H \in \Lambda_r \SU$.
Clearly
$$
    h_{L',\alpha_0} \breve F_\l (z) h^{-1}_{L' (z),\alpha_0}  \in  \Lambda_r \SL (\C)
$$
is holomorphic on $A_r$ away from $\alpha_0$ and $1/{\bar \alpha_0}$, and $\SU$-valued on $\SY^1$. At the roots $\alpha_0$ and $1/{\bar \alpha_0}$,
we have simple poles and we have to study the residues of
$$
    G_\l(z)=\pi_{L'}^{-1}(\l) \breve F_\l(z) \pi_{L' (z) }(\l).
$$
Now we consider the simple factor $L'(z)={ {}^t \overline{\breve F} }_{\alpha_0}(z) L'$.
Let $(L'(z),L'(z)^{\perp})$ be an orthonormal basis of $\C^2$. Note that $L'(z)={}^t \bar{\breve F}_{\alpha_0} (z) L' = \breve F ^{-1}_{1/\bar{ \alpha}_0} (z) L'$ and
$\breve F_{\alpha_0}(z) ^{-1}  L'^{\perp}=L'(z)^{\perp}$. When  $\l \rightarrow 1/\bar \alpha_0$, we have
$$
    \lim _{\l \rightarrow 1/\bar \alpha_0}G_\l L'(z)= \lim _{\l \rightarrow 1/\bar \alpha_0} \sqrt{\tfrac{\lambda-\alpha_0}{1-\bar{\alpha}_0\,\lambda}} Q_{L'} \pi_{\alpha_0}^{-1} Q_{L'}^{-1}  \breve F_\l \bar {\breve F}^t_{\alpha_0}  L' (z) = L'\,,
$$
$$
    \lim _{\l \rightarrow \bar\alpha_0^{-1}} (1-\bar\alpha_0 \l) G_\l L'^{\perp} (z)= \lim _{\l \rightarrow \bar\alpha_0^{-1}}  (1-\bar \alpha_0 \l) \sqrt{\tfrac{1-\bar{\alpha}_0\,\lambda}{\lambda-\alpha_0} }Q_{L'} \pi^{-1}_{\alpha_0} Q_{L'}^{-1} \breve F_\l L '^{\perp}(z) = 0\,.
$$
When $\l \rightarrow  \alpha_0$, we compute
$$
    \lim _{\l \rightarrow  \alpha_0} (\l - \alpha_0) G_\l L' (z) = \lim _{\l \rightarrow  \alpha_0} (\l - \alpha_0) \sqrt{\tfrac{\lambda-\alpha_0}{1-\bar{\alpha}_0\,\lambda}}Q_{L'} \pi_{\alpha_0}^{-1} Q_{L'}^{-1} \breve F_\l    L' (z) = 0\,,
$$
$$
    \lim _{\l \rightarrow  \alpha_0}  G_\l L'^{\perp} (z)= \lim _{\l \rightarrow  \alpha_0}  \sqrt{\tfrac{1-\bar{\alpha}_0\,\lambda}{\lambda-\alpha_0}}Q_{L'} \pi_{\alpha_0}^{-1} Q_{L'}^{-1}  \breve F_\l   \breve F_{\alpha_0}^{-1} L'^{\perp}(z) = L'^{\perp}\,.
$$
This proves the proposition.
\end{proof}

\bibliographystyle{amsplain}

\begin{thebibliography}{10}

\bibitem{Abr}
U.~Abresch, \emph{Constant mean curvature tori in terms of elliptic functions},
  J. Reine U. Angew Math. \textbf{374} (1987), 169--192.

\bibitem{Bob:tor}
A.~I. Bobenko, \emph{All constant mean curvature tori in {$\mathbb{R}^3$},
  {$\mathbb{S}^3$}, {$\mathbb{H}^3$} in terms of theta-functions}, Math. Ann.
  \textbf{290} (1991), 209--245.

\bibitem{Bob:cmc}
\bysame, \emph{Constant mean curvature surfaces and integrable equations},
  Russian Math. Surveys \textbf{46} (1991), 1--45.

\bibitem{BurFPP}
F.~E. Burstall, D.~Ferus, F.~Pedit, and U.~Pinkall, \emph{Harmonic tori in
  symmetric spaces and commuting {H}amiltonian systems on loop algebras}, Ann.
  of Math. \textbf{138} (1993), 173--212.

\bibitem{BurP_adl}
F.~E. Burstall and F.~Pedit, \emph{Harmonic maps via {A}dler-{K}ostant-{S}ymes
  theory}, Harmonic maps and integrable systems, Aspects of Mathematics, vol.
  E23, Vieweg, 1994.

\bibitem{BurP:dre}
\bysame, \emph{Dressing orbits of harmonic maps}, Duke Math. J. \textbf{80}
  (1995), no.~2, 353--382.

\bibitem{DorH:per}
J.~Dorfmeister and G.~Haak, \emph{On constant mean curvature surfaces with
  periodic metric}, Pacific J. Math. \textbf{182} (1998), 229--287.

\bibitem{DorPW}
J.~Dorfmeister, F.~Pedit, and H.~Wu, \emph{Weierstrass type representation of
  harmonic maps into symmetric spaces}, Comm. Anal. Geom. \textbf{6} (1998),
  no.~4, 633--668.

\bibitem{FerPPS}
D.~Ferus, F.~Pedit, U.~Pinkall, and I.~Sterling, \emph{{M}inimal tori in
  {$\mathbb{S}^4$}}, J. reine angew. Math. \textbf{429} (1992), 1--47.

\bibitem{ForW}
A.P. Fordy and J.C. Wood (eds.), \emph{Harmonic maps and integrable systems},
  Aspects of Mathematics, E23, Friedr. Vieweg \& Sohn, Braunschweig, 1994.

\bibitem{Forster-RS}
O.~Forster, \emph{Lectures on {R}iemann surfaces}, Graduate Texts in
  Mathematics, vol.~81, Springer-Verlag, New York, 1991.

\bibitem{hauswirth2006}
L.~Hauswirth, \emph{Minimal surfaces of {R}iemann type in three-dimensional
  product manifolds}, Pacific J. Math. \textbf{224} (2006), no.~1, 91--117.

\bibitem{hks2}
L.~Hauswirth, M.~Kilian, and M.~U. Schmidt, \emph{Properly
  embedded minimal annuli in $\mathbb{S}^2 \times \mathbb{R}$}, arXiv:1210.5953, 2012.

\bibitem{Hit:tor}
N.~Hitchin, \emph{Harmonic maps from a 2-torus to the 3-sphere}, J.
  Differential Geom. \textbf{31} (1990), no.~3, 627--710.

\bibitem{hoffman-white2011}
D.~Hoffman and B.~White, \emph{Axial minimal surfaces in {$S^2\times\bbR$}
  are helicoidal}, J. Differential Geom. \textbf{87} (2011), no.~3, 515--523.

\bibitem{Kil:bub}
M.~Kilian, \emph{Bubbletons are not embedded}, Osaka Jour. Math \textbf{49}
  (2012), no.~3, 653--663.

\bibitem{KilSS}
M.~Kilian, N.~Schmitt, and I.~Sterling, \emph{Dressing {CMC} n-{N}oids}, Math.
  Z. \textbf{246} (2004), no.~3, 501--519.

\bibitem{Kob:bub}
S.-P. Kobayashi, \emph{Bubbletons in 3-dimensional space forms}, Balkan J.
  Geom. Appl. \textbf{9} (2004), no.~1, 44--68.

\bibitem{MPR}
W. H. Meeks III, J. Perez, A. Ros, \emph{Properly embedded minimal planar domain}, preprint.

\bibitem{MRstable}
W. H. Meeks III and H. Rosenberg, \emph{Stable minimal surfaces in $M \times \R$ }J. Differential Geom. \textbf{68}, Number 3 (2004), no.~ 515-534.

\bibitem{MR} W. H. Meeks III and H. Rosenberg, \emph{The theory of minimal surfaces in $M \times \R$}, Comment. Math. Helv. \textbf{80} (2005), no. 4, 811–-858.

\bibitem{McI}
I.~McIntosh, \emph{Global solutions of the elliptic 2d periodic {T}oda
  lattice}, Nonlinearity \textbf{7} (1994), no.~1, 85--108.

\bibitem{McI:tor}
\bysame, \emph{Harmonic tori and their spectral data}, Surveys on geometry and
  integrable systems, Adv. Stud. Pure Math., vol.~51, Math. Soc. Japan, Tokyo,
  2008, pp.~285--314.

\bibitem{PinS}
U.~Pinkall and I.~Sterling, \emph{On the classification of constant mean
  curvature tori}, Ann. Math. \textbf{130} (1989), 407--451.

\bibitem{Poh}
K.~Pohlmeyer, \emph{Integrable hamiltonian systems and interaction through
  quadratic constraints}, Comm. Math. Phys. \textbf{46} (1976), 207--221.

\bibitem{PreS}
A.~Pressley and G.~Segal, \emph{Loop groups}, Oxford Science Monographs, Oxford
  Science Publications, 1988.

\bibitem{rosenberg2002}
H.~Rosenberg, \emph{Minimal surfaces in {${\bbM}^2\times \bbR$}}, Illinois
  J. Math. \textbf{46} (2002), no.~4, 1177--1195.

\bibitem{SteW:bub}
I.~Sterling and H.~Wente, \emph{Existence and classification of constant mean
  curvature multibubbletons of finite and infinite type}, Indiana Univ. Math.
  J. \textbf{42} (1993), no.~4, 1239--1266.

\bibitem{Symes_80}
W.~W. Symes, \emph{Systems of {T}oda type, inverse spectral problems, and
  representation theory}, Invent. Math. \textbf{59} (1980), no.~1, 13--51.

\bibitem{TerU}
C.~Terng and K.~Uhlenbeck, \emph{B\"{a}cklund transformations and loop group
  actions}, Comm. Pure and Appl. Math \textbf{LIII} (2000), 1--75.

\bibitem{Uhl}
K.~Uhlenbeck, \emph{Harmonic maps into lie groups (classical solutions of the
  chiral model)}, J. Diff. Geom. \textbf{30} (1989), 1--50.

\end{thebibliography}

\def\cydot{\leavevmode\raise.4ex\hbox{.}} \def\cprime{$'$}
\providecommand{\bysame}{\leavevmode\hbox to3em{\hrulefill}\thinspace}
\providecommand{\MR}{\relax\ifhmode\unskip\space\fi MR }
\providecommand{\MRhref}[2]{%
\href{http://www.ams.org/mathscinet-getitem?mr=#1}{#2}
}
\providecommand{\href}[2]{#2}

\end{document}